\documentclass[11pt,reqno]{amsart}

\usepackage{lipsum}

\usepackage{graphicx}
\usepackage{array}
\usepackage{amssymb}
\usepackage{hyperref}
\usepackage{amsthm}
\usepackage{amsfonts}
\usepackage{amsmath}
\usepackage{bm}
\usepackage{mathrsfs}
\usepackage[all]{xy}
\usepackage{color}
\usepackage{subfigure}
\usepackage{tikz, tikz-cd}
\usepackage{enumerate}
\usepackage{anysize}
\usepackage{amscd}
\usepackage[letterpaper]{geometry}
\usepackage{geometry}
\usepackage{multirow}
\usepackage{array}
\usepackage{booktabs}
\usepackage{tikz-cd}

\newtheorem{theorem}{Theorem}[section]
\newtheorem{proposition}[theorem]{Proposition}

\newtheorem{lemma}[theorem]{Lemma}

\newtheorem{corollary}[theorem]{Corollary}

\theoremstyle{remark}

\theoremstyle{definition}
\newtheorem{definition}[theorem]{Definition}
\newtheorem{remark}[theorem]{Remark}

\newtheorem{condition}[theorem]{Condition}

\numberwithin{equation}{section}
\numberwithin{figure}{section}
\numberwithin{table}{section}

\newcommand{\dN}{\mathbb{N}}

\newcommand{\RR}{\mathbb{R}}
\newcommand{\ZZ}{\mathbb{Z}}

\newcommand{\p}{\partial}

\newcommand{\s}{\sigma}
\newcommand{\hr}{\hat{r}}
\newcommand{\n}{\nabla}

\allowdisplaybreaks[3]

\begin{document}

 \title[Conformally prescribed scalar curvature on orbifolds]{Conformally prescribed scalar curvature on orbifolds}

 \author{Tao Ju}\thanks{The authors were partially supported by NSF Grants DMS-1811096 and DMS-2105478.}
 \author{Jeff Viaclovsky}
 \address{Department of Mathematics, University of California, Irvine, CA 92697}
 \email{jut1@uci.edu, jviaclov@uci.edu}

\begin{abstract}
We study the prescribed scalar curvature problem in a conformal class on orbifolds with isolated singularities. We prove a compactness theorem in dimension $4$, and an existence theorem which holds in dimensions $n \geq 4$. This problem is more subtle than the manifold case since the positive mass theorem does not hold for ALE metrics in general.  We also determine the $\rm{U}(2)$-invariant Leray-Schauder degree for a family of negative-mass orbifolds found by LeBrun.  
\end{abstract}
\maketitle

\setcounter{tocdepth}{1}
\tableofcontents

\section{Introduction}
\label{intro}
To begin, we give the definition of a Riemannian orbifold.
\begin{definition}
\label{obf_def}
We say that $(M,g)$ is a Riemannian orbifold of dimension $n$ if $M$ is a smooth manifold of dimension $n$ with a smooth Riemannian metric away from a finite singular set
\begin{align}
\Sigma = \{(q_1, \dots,  q_l)\}.
\end{align}
Near each singular point $q_j$, there exists a neighborhood $U_j$ of $q_j$, a nontrivial finite subgroup 
$\Gamma_j \subset {\rm{O}}(n)$ acting freely on $\RR^{n} \setminus \{0\}$ 
and a $\Gamma_j$-equivariant diffeomorphism  
\begin{align}
\varphi_j: \widetilde{U_j}\to B_{\s_j}(0),
\end{align}
where $\widetilde{U_j}$ is the completion (by adding a point $\tilde{q}_j$) of the universal cover of $U_j \setminus \{q_j\}$,
and $ B_{\s_j}(0)$ is a ball of radius $\sigma_j$ about the origin in $\RR^4$. 
Furthermore, $(\varphi_j)_*  \pi_j^* g$ extends to a smooth Riemannian metric on $B_{\s_j}(0)$,
where $\pi_j :  \widetilde{U_j} \setminus \{\tilde{q}_j\} \rightarrow U_j \setminus \{q_j\}$ is a universal covering map. 
\end{definition}
Our convention is that if $l = 0$, then $\Sigma = \emptyset$, and $(M,g)$ is a smooth Riemannian manifold. But if $l \geq 1$, then there must be nontrivial singular points. 
We next define the meaning of a $C^k(M)$ function on a Riemannian orbifold $(M,g)$.
\begin{definition}
A function $f: M \rightarrow \RR$ is in $C^k(M)$ if $f: M \setminus \Sigma \rightarrow \RR$ is in $C^k( M \setminus \Sigma)$,
and near each singularity, there exists a coordinate system $\varphi_j$ such that the function $f  \circ \pi_j \circ \varphi_j^{-1} : B_{\s_j}(0) \rightarrow \RR$ is in $C^k( B_{\s_j}(0))$.
We can also define the spaces $C^{\infty}(M), C^{k,\alpha}(M), C^k_{loc}(M)$ in a similar fashion. 
\end{definition}
Note that since linear terms are never invariant under a nontrivial orbifold group, a $C^1$ function necessarily has a critical point at any nontrivial orbifold singularity. In other words,  $\Sigma \subset Crit(f)$, where $Crit(f) \equiv \{x \in M \ | \ \nabla_g f(x) = 0\}$. 

\begin{remark} In general, an orbifold can have higher-dimensional singular sets. So our definition is restrictive in that we only allow isolated quotient singularities. 
\end{remark}

Assume $(M, g)$ is a compact Riemannian $n$-orbifold 
with positive scalar curvature $R_g > 0$. Let $K > 0$ be a positive $C^2$ function on $M$. We will study the following equation
\begin{align}
\label{yamabeequ}
-\Delta_g u + c(n)R_gu = K u^{p},
\end{align}
where $c(n)=\frac{n-2}{4(n-1)}$, $1<p\leq \frac{n+2}{n-2}$ and $\Delta_g$ is the Laplacian operator associated with $g$. Let $L_g = \Delta_g - c(n)R_g$ denote the conformal Laplacian of the metric $g$. When $p=\frac{n+2}{n-2}$, the solution of equation \eqref{yamabeequ} corresponds to the prescribed scalar curvature problem. That is, the metric given by $\tilde g = u^{\frac{4}{n-2}} g$ has scalar curvature $R_{\tilde{g}} = \frac{4(n-1)}{n-2}K$.

The Yamabe problem on manifolds is well-understood, and we refer to \cite{Aubin3, Lee-Parker, Schoen, Yamabe} 
and references therein.  
The prescribed scalar curvature problem on $S^n$ and other manifolds has been studied in many works; see for example \cite{BA, CGY, ChenLin, YanyanI, YanyanII, Malchiodi, SchoenZhang}.
Prescribed scalar curvature on manifolds is studied for example in \cite{Duke, MalchiodiMayer, Mayer}. More recently, the Yamabe problem on singular spaces has been of interest; see for example \cite{ABI, ABII, ACMI, ACMII, CLV, MondelloI, MondelloII, monopole_V, Viaclovsky_Tohoku}.

Analogous to the generalization from manifolds to orbifolds, there is the following generalization of asymptotically flat (AF) metrics. 
\begin{definition}
\label{ALEdef}
 A complete Riemannian orbifold $(X^n,g)$ with finitely many singular points
is called {\em{asymptotically locally 
Euclidean}} or {\rm{ALE}} of order $\tau$ if 
it has finitely many ends and for each end
there exists a finite subgroup 
$\Gamma \subset {\rm{O}}(n)$ 
acting freely on $\RR^n \setminus \{0\}$ and a diffeomorphism 
$\psi : X \setminus K \rightarrow ( \mathbb{R}^n \setminus \overline{B_R(0)}) / \Gamma$ 
where $K$ is a compact subset of $X$, and such that under this identification, 
\begin{align}
\label{ale1}
(\psi_* g)_{ij} &= \delta_{ij} + O( \rho^{-\tau}),\\
\label{ale2}
\ \partial^{k} (\psi_*g)_{ij} &= O(\rho^{-\tau - |k| }),
\end{align}
for any partial derivative of order $|k|$, as
$\rho \rightarrow \infty$, where $\rho$ is the distance to some fixed basepoint.  
\end{definition}
We will occasionally refer to an AF metric as an ALE metric, since
AF is exactly the case of ALE with $\Gamma = \{ e \}$. Next, we give the definition of the ADM mass on asymptotically locally Euclidean (ALE) orbifolds.
\begin{definition}
\label{mass_def}
Given an $n$-dimensional ALE orbifold $(X,g)$ with asymptotic coordinates $\{z^i\}$ and quotient group $\Gamma$ near $\infty$, define the ADM mass as follows:
\begin{align}
\label{mass_form}
m(g) =\frac{|\Gamma|}{Vol(S^{n-1})}\lim_{r\to\infty}\int_{S_r/\Gamma}\sum_{i,j=1}^n(\p_ig_{ij}-\p_jg_{ii}) ( \p_j \lrcorner dV_z),
\end{align}
where $S_r/\Gamma$ is the hypersurface at $|z| = r$, and $dV_z$ is the Euclidean volume element.
\end{definition}
\begin{remark}  Bartnik proved that in the AF case, if $\tau > (n-2)/2$ then the mass is well-defined and independent of the choice of coordinates at infinity \cite{Bartnik}. A similar argument shows that the same result holds in the ALE case. 
\end{remark}

\begin{remark} Note that if $|\Gamma|=\{e\}$ is the trivial group, the formula above defines the mass for AF orbifold $(X,g)$, which is consistent with \cite[Definition 8.2]{Lee-Parker}.
Note also that if $\Gamma \neq \{e\}$, our coefficient differs from \cite{HeinLeBrun} due to the factor of $|\Gamma|$. Our convention has the advantage that it eliminates the need for writing extra factors of $|\Gamma|$ in several formulas. 
\end{remark}

For any Riemannian orbifold $(M,g)$ with positive scalar curvature, we can construct a scalar-flat ALE orbifold $(X,\hat{g})$ by the following well-known procedure.  
\begin{definition}
\label{def_psi}
Take $g$-normal coordinates $\{x^i\}$ centered at point $\bar x$ and let $r=|x|$ denote the distance function. Let $\psi_{\bar{x}} > 0$ be the Green's function of $L_g$ with leading term $r^{2-n}$ near $\bar x$. Then $(X_{\bar x}, \hat g_{\bar{x}}) = (M \setminus \bar{x}, \psi_{\bar x}^{4/(n-2)} g)$ is a scalar flat ALE orbifold. 
We will refer to $(X_{\bar x}, \hat{g}_{\bar{x}})$ as the conformal blow-up of $g$ at the point $\bar x$.
\end{definition}
Of course, if $\bar{x}$ is a smooth point of $M$, then $(X_{\bar x}, \hat{g}_{\bar{x}})$ is an AF orbifold, but if $\bar x$ is a singular point, then $(X_{\bar x}, \hat{g}_{\bar x})$ is an ALE orbifold.
\subsection{Positive Mass Theorem on AF orbifolds}
Our first result is that the positive mass theorem does hold for AF orbifolds. 
\begin{theorem}
\label{orbifold_PMT}
Let $(X, g_X)$ be an asymptotically flat (AF) $n$-dimensional Riemannian orbifold with finitely many isolated singular points, with $R(g_X) \geq 0$,
and which is of order $\tau_X > \frac{n-2}{2}$.
Then $\mathrm{mass}(g_X) \geq 0$.
and $\mathrm{mass}(g_X) = 0$ if and only if
there are no nontrivial orbifold singularities, and $(X,g_X)$ is isometric to $(\RR^{n}, g_{Euc})$.
\end{theorem}
This is proved in Section~\ref{s:PMT}. The basic idea is to use a certain Green's function for the Laplacian based at the orbifold points, which we use to reduce to the positive mass theorem for manifolds with concave boundary due to Hirsch-Miao \cite{HirschMiao} using the fundamental work of Schoen-Yau \cite{SYI, SYII, SchoenYauPMT}.

\begin{remark}
The positive mass theorem holds on ALE spaces only with some extra assumptions; see for example \cite{Nakajima}. In contrast, the positive mass theorem does not necessarily hold for arbitrary ALE metrics with nonnegative scalar curvature. The first examples were given in \cite{LeBrun}, and many more in \cite{HeinLeBrun}.
Theorem \ref{orbifold_PMT} shows that the study of the Yamabe equation on orbifolds can be substantially different from the manifold case. In dimension $4$, we can define a mass function $m:  M  \rightarrow \RR$ by assigning the mass of the conformal blow-up at $\bar{x} \in M$. In the manifold case, this is a smooth function \cite{Habermann}. However, the above result shows that if $(M,g)$ is an orbifold and the conformal blow-up has negative mass at $\bar x \in \Sigma$, then the mass function $m$ is necessarily \textit{discontinuous} at $\bar x$; see Corollary~\ref{c:discont} below. 
\end{remark}

\subsection{Existence and compactness results}
There is a long history of existence and compactness results for the Yamabe problem, which is the case when $K = constant$. The fundamental idea for compactness was due to Schoen \cite{Schoen1987, Schoen_csc, Schoen_Courses}; 
 compactness results for the Yamabe problem in low dimensions were then proved in \cite{Druet, KMS, Li-Zhang, Li-Zhu, Marques}. As mentioned above, existence and compactness results for variable $K$ have been studied in great detail on $S^n$ and on manifolds. 
Our next result generalizes many of these results in dimension four to 
the case of a Riemannian orbifold.  
\begin{theorem}\label{cpt_crit_thm_combined}
Let $(M,g)$ be a compact Riemannian 4-dimensional orbifold with positive scalar curvature. Let $K\in C^2(M)$ satisfy 
\begin{align}
0< &\delta_1 \equiv \inf_M K, \\
\label{mass_ineq}
0 < &\delta_2 \equiv \inf_{ \bar{x} \in Crit(K)} \Big|m(\hat g_{\bar{x}}) + \frac{\Delta_g K(\bar x)}{2K(\bar x)}
\Big|.
\end{align}
Then there exists some constant $C$ depending only on $M, g, \delta_1,\delta_2, \Vert K\Vert_{C^2(M)}$ such that 
\begin{align}
\label{comp_est}
1/C\leq u\leq C\mathrm{\ \ and\ \ }\Vert u\Vert_{C^{2,\alpha}(M)}\leq C
\end{align}
for all solutions $u$ of \eqref{yamabeequ} with $p = 3$, where $0<\alpha<1$. Moreover, if 
\begin{align}
\label{mass_ineq_sub}
0< \delta_3 \equiv \inf_{ \bar{x} \in Crit(K)} \Big( m(\hat g_{\bar{x}}) + \frac{\Delta_g K(\bar x)}{2K(\bar x)}\Big),
\end{align}
then \eqref{comp_est} holds 
for all  $1< 1+\varepsilon <p\leq 3$ where $C$ in addition depends upon $\delta_3,\epsilon$. Consequently, in this case there exists a solution
$u$ of \eqref{yamabeequ} with $p = 3$. 
\end{theorem}
Theorem \ref{cpt_crit_thm_combined} will be proved in Sections \ref{prelim}, \ref{local_ana}, and \ref{s:blow-up}. One of the main difficulties is due to the ``discontinuity'' of conformal normal coordinates at an orbifold point. In Section~\ref{s:blow-up}, we will show that if concentration of a sequence $u_k$ happens at an orbifold point, then the local maxima of $u_k$ must occur \textit{exactly} at the orbifold point for $k$ sufficiently large; see Condition~\ref{blow-up_cond} which is a generalization of the isolated simple blow-up condition. Once we prove that blow-up points satisfy Condition \ref{blow-up_cond};
we can then fix the coordinates to be centered at the orbifold point, and then there is a contribution only from the mass function at the blow-up point. Note: our argument does not need continuity of the mass function which, as pointed out above, is not true in general anyway.

\begin{remark} We emphasize that under the assumption \eqref{mass_ineq}, we are only claiming compactness; we do not have the existence in this general case. The main reason is that the subcritical method only works under the stronger assumption \eqref{mass_ineq_sub}. 
However, the compactness result does allow one to define the Leray-Schauder degree. 
We expect that by imposing extra assumptions on $K$, one should be able to prove this degree is non-zero in some cases where \eqref{mass_ineq_sub} does not hold, but we do not pursue this here.  Results along this line in the manifold case can be found in \cite{Duke, YanyanII, MalchiodiMayer, Mayer}. 
\end{remark}

\subsection{Variational methods}
Next, we have another existence result in dimensions $n \geq 4$ which is proved using variational methods. Roughly, this says that if $K$ is not too large away from the orbifold points, then we need assumption \eqref{mass_ineq_sub} to hold at only one orbifold point. 
\begin{theorem}\label{energy_existence}
Let $(M, g)$ be a compact Riemannian $n$-dimensional orbifold with singularities $\Sigma_\Gamma = \{(q_1,\Gamma_1),\cdots,(q_l,\Gamma_l)\}$ and positive scalar curvature. Let $K$ be a positive smooth function on $M$. 
Assume that
\begin{align}
\label{maxcond}
\max_{1 \leq i \leq l} \{ |\Gamma_i|^{\frac{2}{n-2}} K(q_i) \} \geq \sup K,
\end{align}
and
\begin{equation}
\label{energy_exist_cond}
\begin{cases}
m(\hat g_{q_{i_0}}) + \frac{\Delta_g K(q_{i_0})}{2 K(q_{i_0})} > 0, & n=4,\\
\Delta_g K(q_{i_0}) > 0, & n \geq 5,
\end{cases}
\end{equation}
for some $i_0$ such that $|\Gamma_{i_0}|^{\frac{2}{n-2}} K(q_{i_0}) = \max_{1 \leq i \leq l} \{ |\Gamma_i|^{\frac{2}{n-2}} K(q_i) \} $,
where $\hat g_{q_{i_0}}$ is as in Definition~\ref{def_psi}.
Then there exists a positive smooth solution of equation \eqref{yamabeequ} with $p = \frac{n+2}{n-2}$.
\end{theorem}
This will be proved in Section \ref{s:Energy}, which is closely related to \cite[Theorem~2.1]{E-S}, \cite[Theorem~3.1]{Akutagawa} and \cite[Proposition 5.1]{M-M}. The analogue on manifolds is that \eqref{energy_exist_cond} has to hold at one maximum point of $K$, which is only possible in dimension $n=4$. Hence the theorem is quite special to the orbifold case in dimensions $n\geq 5$. We next apply this to some examples. 

 Consider $S^n \subset \RR^{n+1}$, and let $\Gamma \subset {\rm{SO}}(n)$ be a finite subgroup acting freely on $S^{n-1} \subset S^n \cap \{x_{n+1} = 0\}$. 
Then $S^n / \Gamma$ with the spherical metric $g_S$ is a Riemannian orbifold with 2 orbifold points $q_1$ and $q_2$.
\begin{corollary}[Spherical football orbifold]
With $\Gamma$ as above, let $K : S^n/\Gamma \rightarrow \RR_+$ be a smooth function and without loss a generality assume that $K(q_1) \geq K(q_2)$. If
\begin{align} 
|\Gamma|^{\frac{2}{n-2}} K(q_1) &\geq \sup K,\\ 
\Delta_{g_S} K (q_1) &> 0,
\end{align}
then there exists a positive smooth solution of equation \eqref{yamabeequ} with $g=g_S$ and $p = \frac{n+2}{n-2}$.
\end{corollary}
We note that for $|\Gamma| = 2$, this result also follows from \cite[Theorem 1.12]{Leung-Zhou}.
Another non-trivial example to which Theorem \ref{energy_existence} applies is to the Calabi metric $g_{CAL(n)}$, which is a $U(n)$-invariant Ricci-flat K\"ahler ALE metric on the total space $X_n$ of the line bundle $\mathcal O(-n)\to \mathbb P^{n-1}$. For details on the Calabi metric, see for example~\cite{M-V}.
\begin{corollary}[Calabi orbifold]
\label{c:Calabi}
For $n\geq 2$, take a $U(n)$-invariant conformal compactification $(\check X_n, \check g_{CAL(n)})$ of the Calabi metric $(X_n, g_{CAL(n)})$, such that the infinity of $g_{CAL(n)}$ is compactified to an orbifold point $\check q$, with quotient group $\mathbb Z/n \mathbb Z$. Let $K$ be a positive smooth function on $\check X_n$. Assume that
\begin{align}
n^{\frac{1}{n-1}}K(\check q) \geq& \sup K,\\
\label{laplacian>0}
\Delta_{\check{g}_{CAL(n)}} K(\check q) >& 0,
\end{align}
then there exists a positive smooth solution of equation \eqref{yamabeequ} with $g = \check g_{CAL(n)}$ and $p = \frac{n+1}{n-1}$.
\end{corollary}

\begin{remark} In this case, from Theorem~1.3 and Remark~4.7 in \cite{monopole_V}, we know that there is no solution with $K = constant > 0$, which certainly does not contradict with the condition on the Laplacian at the orbifold point.  
However, Corollary \ref{c:Calabi} implies that any sufficiently small perturbation of a constant satisfying
$\Delta_{\check g_{CAL(n)}}K (\check{q}) > 0$ \textit{does} admit a solution. 
So the case of $K = constant$ is right on the boundary of existence. Furthermore, if we take a sequence of functions $K_k$ satisfying
\eqref{laplacian>0}, but converging to a positive constant in $C^2$ norm, then this gives an example where bubbling \textit{must} occur as $k \to \infty$. 
\end{remark}

 For $n=2$, $(X_2, g_{CAL(2)})$ is also known as the Eguchi-Hanson metric, which is included in the family of LeBrun metrics, which we discuss next.  

\subsection{LeBrun metrics}\label{lebrunsub} In this subsection, we will use the previous results to analyze the prescribed scalar curvature problem on the family of orbifolds mentioned above: the LeBrun negative mass metrics on $\mathcal{O}_{\mathbb{P}^1}(-n)$. Note: these are $4$-dimensional manifolds, so in this subsection the integer $n$ will instead be used to denote $-c_1$, where $c_1$ is the first Chern class of this line bundle. These metrics possess an isometric $\mathrm{U}(2)$-action. As we will see below, we have some results in the general case, but our most complete results are under a $\mathrm{U}(2)$-symmetry assumption. 

In \cite{LeBrun}, LeBrun presented the first known example of a scalar-flat ALE metric of negative ADM mass. For any $n\in \dN^*$, define
\begin{align}
g_{LEB(n)} =\frac{1+\hr^2}{n+\hr^2}d\hr^2+(1+\hr^2)\s_1^2+(1+\hr^2)\s_2^2+\frac{\hr^2(n+\hr^2)}{1+\hr^2}\s_3^2,
\end{align}
where $\hr\in (0,\infty)$ is the radial coordinate, and $\{\s_1,\s_2,\s_3\}$ is a left-invariant coframe on $S^3=SU(2)$. Attach a $\mathbb{P}^1$ at $\hr = 0$. After taking a quotient by $\Gamma_n = \ZZ/n\ZZ$, the metric extends smoothly over this $\mathbb{P}^1$. Furthermore, the metric is ALE with group action $\Gamma_n$ near $\hr =\infty$ and the resulting manifold is diffeomorphic to $\mathcal{O}_{\mathbb {P}^1}(-n)$.

The mass term defined in \cite{LeBrun} differs from our Definition \ref{mass_def} by a scaling factor. With a modification of the mass computed in \cite{LeBrun}, we obtain
\begin{align}
m(g_{LEB(n)}) = -2(n-2).
\end{align}
For any $n\in \dN^*$, $g_{LEB(n)}$ is scalar-flat. Note that $g_{LEB(1)}$ is conformal to the Fubini-Study metric on $\mathbb{P}^2$, and is known as the Burns metric, which is AF and has positive mass. As pointed out above, $g_{LEB(2)}$ is the Eguchi-Hanson metric, which is Ricci-flat ALE and has zero mass. For any $n\geq 3$, $g_{LEB(n)}$ has negative mass. More details and properties of LeBrun metrics can be found in \cite{D-L, LeBrun, monopole_V}.
 
We next choose an orbifold compactification by defining
\begin{align}
\check g_{LEB(n)} = \frac{1}{(n+\hr^2)^2}\cdot g_{LEB(n)}.
\end{align}
Then $(\check{\mathcal{O}}_{\mathbb {P}^1}(-n),  \check g_{LEB(n)})$ is a compact orbifold with singular point $\check{q}$ at $\hr=\infty$ with quotient group $\Gamma_n$. Its scalar curvature is computed to be
\begin{align}
R_{\check g_{LEB(n)}} = \frac{24n(n + \hr^2)}{1+\hr^2} > 0.
\end{align}
In this case, Theorem \ref{energy_existence} specializes to the following. 

\begin{theorem} 
\label{t:LEB1}Consider $(\check{\mathcal{O}}_{\mathbb {P}^1}(-n),  \check g_{LEB(n)})$ for $n \in \mathbb{N}^*$. Let $K$ be a positive smooth function such that 
\begin{align}
\begin{split}
\sup K &\leq n K(\check{q}),\\
4 (n -2)K(\check{q}) &< \Delta_{\check g_{LEB(n)}} K(\check{q}).
\end{split}
\end{align}
Then there exists a positive smooth solution of equation \eqref{yamabeequ} with $g=\check g_{LEB(n)}$ and $p = 3$.
\end{theorem}
Next, we turn to the U$(2)$-invariant problem. Define a function space $\mathcal X_n$
\begin{align}
\label{X_def}
\mathcal X_n = \{ K : \check{\mathcal{O}}_{\mathbb {P}^1}(-n) \rightarrow \RR \ | 
\ K \mbox{ is smooth, U}(2)\mbox{-invariant, and } K > 0 \}.
\end{align}
Given any $K \in\mathcal X_n$, we are asking whether there exists a solution $u \in\mathcal X_n$ of the equation 
\begin{align}
\label{leb_conf}
-\Delta_{\check g_{LEB(n)}}u + \frac{1}{6}R_{\check g_{LEB(n)}}u = Ku^3.
\end{align}
Decompose $\mathcal X_n$ into three disjoint subsets
\begin{align}
\label{X_decomp}
\begin{split}
\mathcal X_{n, +} &=  \{K\in\mathcal X_n: \Delta_{\check g_{LEB(n)}} K(\check{q})-4 (n -2)K(\check{q}) > 0 \},\\
\mathcal X_{n, 0} &=  \{K\in\mathcal X_n: \Delta_{\check g_{LEB(n)}} K(\check{q})-4 (n -2)K(\check{q}) = 0 \},\\
\mathcal X_{n, -} &=  \{K\in\mathcal X_n: \Delta_{\check g_{LEB(n)}} K(\check{q})-4 (n -2)K(\check{q}) < 0 \}.
\end{split}
\end{align}
Define a function space
\begin{align}
\label{Omega_def}
\Omega_{\Lambda, n} = \{u\in \mathcal X_n: \Vert u\Vert_{C^{2,\alpha}(\check{\mathcal{O}}_{\mathbb {P}^1}(-n))} < \Lambda,\ u > \Lambda^{-1}\},
\end{align}
and a map
\begin{align}
\label{F_map_def}
F_{p,K,n}:\bar\Omega_{\Lambda, n}\to C^{2,\alpha}(\check{\mathcal{O}}_{\mathbb {P}^1}(-n))\ \ \mathrm{by\ \ }F_{p,K,n}(u) = u + L_{\check g_{LEB(n)}}^{-1}(Ku^p),
\end{align}
where $L_{\check g_{LEB(n)}} = \Delta_{\check g_{LEB(n)}} - \frac{1}{6}R_{\check g_{LEB(n)}}$.
Similar to \cite{Schoen, Li-Zhang, Li-Zhu, KMS, Marques}, we can define the $\mathrm{U}(2)$-invariant Leray-Schauder degree of $F_{p,K,n}$ in the region $\Omega_{\Lambda,n}$ with respect to $0\in C^{2,\alpha}(\check{\mathcal{O}}_{\mathbb {P}^1}(-n))$, which we denote by $\deg_{\mathrm{U}(2)}(F_{p,K,n}, \Omega_{\Lambda, n}, 0)$. Using Theorem \ref{cpt_crit_thm_combined} and some other special techniques, we present our last theorem, which gives a complete understanding of the U$(2)$-invariant Leray-Schauder degree.
\begin{theorem}
\label{lebrun_thm}
For any $n\in\mathbb N^*$,  there exists a constant $C$, depending only on $(\check{\mathcal{O}}_{\mathbb {P}^1}(-n),  \check g_{LEB(n)})$
such that we have the following conclusions.

\noindent
(1) 
For all $K\in \mathcal X_{n, +}$ and $\Lambda > C$, we have
\begin{align}
\deg_{\mathrm{U}(2)}(F_{3,K,n}, \Omega_{\Lambda,n}, 0) = -1.
\end{align}
Consequently, equation \eqref{leb_conf} admits at least one $\mathrm{U}(2)$-invariant solution in $\mathcal X$.

\noindent
(2) 
For all $K\in \mathcal X_{n, -}$ and $\Lambda >C$, we have
\begin{align}
\deg_{\mathrm{U}(2)}(F_{3,K,n}, \Omega_{\Lambda,n}, 0) = 0.
\end{align}
\end{theorem}
\begin{remark}
We note that vanishing of the Leray-Schauder degree does not give any information regarding the existence of a solution. However, we can moreover show that there is no solution at all for a large class of functions in $\mathcal X_{n, -}$; see Theorem~\ref{lebrun_thm_2}.
In particular, there is no solution for $K = constant$ when $n \geq 2$.
For $n =2$ and $K = constant$, nonexistence of {\textit{any}} solution (symmetric or non-symmetric) was proved in \cite{monopole_V}. However, it is still an open question whether the case $K = constant$ possibly admits some non-symmetric solution when $n > 2$.
\end{remark}

\begin{remark} For any $n\in \mathbb N^*$, the set $\mathcal{X}_{n, 0}$ can be viewed as a ``wall'' in the space of positive radial functions $\mathcal{X}_n$, and the Leray-Schauder degree jumps by $1$ upon crossing this wall, which is a phenomenon observed in many other geometric PDE problems.
\end{remark}
All of the above results regarding the LeBrun metrics are proved in Section~\ref{s:LeBrun}.

\subsection{Acknowledgements} The authors would like to thank Richard Schoen for suggesting to consider the Green's function on an AF orbifold to prove Theorem \ref{orbifold_PMT}; this greatly simplified our original proof which was based on a resolution of singularities argument. 

\section{Properties of the mass}
\label{s:PMT}
In this section, we will prove Theorem \ref{orbifold_PMT}. 
For simplicity,  let us assume that there is exactly 1 orbifold point, which we denote as $q$, with orbifold group $\Gamma \subset \mathrm{O}(n)$ (the argument below easily generalizes to the case of multiple orbifold singularities).  
Let $r$ denote a positive smooth function which is the Euclidean distance in the AF coordinate system, and near $q$ is the distance to $q$.
We will first prove a lemma showing existence of a certain harmonic function on $X\setminus \{q\}$.
\begin{lemma}  
\label{lemma:G}
There exists a unique harmonic function $H : X \setminus \{q\} \rightarrow \RR$
which satisfies $H > 1$ and admits the expansion  
\begin{align}
H = 
\begin{cases}
r^{2-n} + O(r^{4-n - \epsilon}) & \mbox{ as } r \to 0\\
1 + A  r^{2-n} + O(r^{2-n - \epsilon})& \mbox{ as } r \to \infty \\
\end{cases},
\end{align}
for $\epsilon > 0$ sufficiently small, for some constant $A > 0$. 
\end{lemma}
\begin{proof}
Let $\phi$ be the cutoff function
\begin{align}
\phi(t) =
\begin{cases}
1 & t \leq 1 \\
0 & t \geq 2\\
\end{cases},
\end{align}
and consider 
\begin{align}
h_0 = \phi( r/r_0) r^{2-n},
\end{align}
for $r_0 > 0$ small. Since $h_0$ is harmonic with respect to the Euclidean metric near point $q$, by expanding $\Delta_g$ at the Euclidean metric, it is not hard to see that
\begin{align}
\label{Lh}
\Delta_g h_0 =
\begin{cases}
O(r^{2-n}) & r \to 0 \\
0 & r \geq 2r_0\\
\end{cases}.
\end{align}
Denote $X_q = X \setminus \{q\}$. The argument below uses weighted H\"older space theory; for background we refer to \cite{Bartnik, Lee-Parker}. Consider the doubly weighted H\"older space $C^{k,\alpha}_{\delta_0, \delta_\infty}(X_q)$, which satisfies if $u \in C^{k,\alpha}_{\delta_0, \delta_\infty}(X_q)$
then 
\begin{align}
u =
\begin{cases}
 O(r^{\delta_0}) & r \to 0 \\
O(r^{\delta_{\infty}}) & r \to \infty\\
\end{cases}.
\end{align}
For any $\epsilon > 0$, from \eqref{Lh}, we have 
\begin{align}
\Delta h_0 \in C^{k-2,\alpha}_{2-n - \epsilon, -n + \epsilon}.
\end{align}
Consider the operator
\begin{align}
\label{Deltaw}
\Delta_g : C^{k,\alpha}_{4-n - \epsilon, 2 - n + \epsilon}(X_q)
\rightarrow  C^{k-2,\alpha}_{2-n - \epsilon, -n + \epsilon}(X_q).
\end{align}
The adjoint operator has domain $ ( C^{k-2,\alpha}_{2-n - \epsilon, -n + \epsilon }(X_q))^*$,
and kernel elements lie in the doubly weighted space
$C^{k,\alpha}_{ -2 + \epsilon,-\epsilon }(X_q)$.
The removable singularity theorem then says that 
a kernel element $u$ in this space extends to $X$, 
and then $u \in C^{k,\alpha}_{-\epsilon}(X)$ (the weighted space on $X$ with only a weight $\delta_{\infty} = - \epsilon$ at infinity), so $u \equiv 0$. 
Thus for $\epsilon >0$ sufficiently small, the operator in \eqref{Deltaw} is surjective,
and we can solve for $\Delta h_{\epsilon} = \Delta_g h_0$, 
with $h_{\epsilon} \in  C^{k,\alpha}_{4-n - \epsilon, 2 - n + \epsilon}(X_q)$.

The function $h \equiv  h_0 - h_{\epsilon} $ satisfies $\Delta_g h =  0$,
and by the existence of a harmonic expansion near $\infty$, 
it admits the expansion.  
\begin{align}
h = 
\begin{cases}
r^{2-n} + O(r^{4-n - \epsilon}) & r \to 0\\
A r^{2-n} + O(r^{2-n - \epsilon})& r \to \infty \\
\end{cases},
\end{align}
for some constant $A$.  We then define $H = 1 + h$,
which is harmonic. We have that 
$\lim_{r \to \infty} H = 1$, and $\lim_{r \to 0} H = + \infty$.
If $H$ were not strictly larger than $1$, then it would have an interior minimum. The  strong maximum principle would then imply that $H$ is constant, which is impossible. So $H > 1$, which clearly implies that $A > 0$. Obviously, $H$ is unique. 
\end{proof}

\begin{proof}[Proof of Theorem \ref{orbifold_PMT}]
For any constant $\delta > 0$, we define 
\begin{align}
H_\delta = \delta H + (1 - \delta),
\end{align}
which satisfies $\Delta_g H_\delta = 0$, $H_\delta > 1$, 
and admits the expansion 
\begin{align}
H_\delta = 
\begin{cases}
\delta r^{2-n} + O(r^{4-n - \epsilon}) & r \to 0\\
1 + \delta A r^{2-n} + O(r^{2-n - \epsilon})& r \to \infty \\
\end{cases},
\end{align}
for $\epsilon > 0$ sufficiently small, for the fixed constant $A$ from Lemma \ref{lemma:G}.

Next, we consider the metric $(X_q, g_\delta) = (X \setminus \{q\}, H_\delta^\frac{4}{n-2} g_X)$. 
Near $r \sim \infty$,  $g_\delta$ has a single AF end
of order $\min\{ \tau_X, n-2\}$.  Since $q$ is an orbifold point, near $q$, $g_\delta$ has a single ALE end of order $\tau = 2 - \epsilon$.  To see this, choose Riemannian normal 
coordinates $\{x^i\}$ for $g_X$ around $q$, then we have the expansions 
\begin{align}
g_X &= dx^2 + O(|x|^{2}),\\
H_\delta &= \delta |x|^{2-n} + O(|x|^{4-n-\epsilon}),
\end{align}
which yield the expansion
\begin{align}
g_\delta = H_\delta^{\frac{4}{n-2}} g_X
= \delta^{\frac{4}{n-2}}|x|^{-4} (1 + O(|x|^{2-\epsilon}))^{\frac{4}{n-2}}
(dx^2 + O(|x|^2)),
\end{align}
as $|x| \to 0$.
Next, define coordinates $y$ by  
\begin{align}
 y = \delta^{\frac{2}{n-2}}\frac{x}{|x|^2}.
\end{align}
A computation then shows that 
\begin{align}
g_\delta = dy^2 + O(|y|^{-2 + \epsilon})
\end{align}
as $|y| \to \infty$, so $g_\delta$ is indeed ALE of order $\tau = 2 - \epsilon$.  
Note also that the scalar curvature of $g_\delta$ is given by 
\begin{align}
R(g_\delta) = c(n)^{-1}H_\delta^{- \frac{n+2}{n-2}} ( -\Delta_g H_\delta + c(n) R_g H_\delta) 
=  H_\delta^{- \frac{4}{n-2}} R_g \geq 0. 
\end{align}
Given $\delta > 0$, we can choose a very large distance sphere $\Sigma_\delta$ in 
the ALE end of $g_\delta$ which is strictly concave with respect to the normal pointing to the AF end. 
Let $X_{\Sigma_\delta}$ be the manifold with boundary obtained by removing the ALE end outside of $\Sigma$, which is a manifold with strictly concave boundary with a single AF end. From \cite[Theorem~1.5 and Remark~1.7]{HirschMiao}, we conclude that 
\begin{align}
m(X_{\Sigma_\delta}, g_\delta) \geq 0.
\end{align}
But an easy computation shows that 
\begin{align}
m(X_{\Sigma_\delta}, g_\delta) = m(X, g_X) + b(n) \delta A,
\end{align}
where $b(n) > 0 $ is a dimensional constant. 
Since this is true for any constant $\delta > 0$, and $A$ is a fixed constant, we conclude
that 
\begin{align}
m(X,g_X) \geq 0. 
\end{align}
Note that if $m(X,g_X) = 0$, then we cannot conclude that 
$m(X_{\Sigma_\delta}, g_\delta) = 0$, since $A > 0$. Therefore we cannot directly use the equality case in \cite[Theorem~1.5]{HirschMiao}. So to finish the proof, if $\mathrm{mass}(X, g_X) =0$ then we instead argue as in \cite[Lemma~10.7]{Lee-Parker} to conclude that $g_X$ is Ricci-flat (this argument is valid in our orbifold setting). Since $g_X$ is asymptotically flat, we have asymptotic equality in Bishop's volume inequality (which holds for orbifolds; see \cite{Borzellino}). This implies that $g_X$ is flat, which clearly implies that there can be no nontrivial orbifold singularities, and $(X,g_X)$ is isometric to $(\RR^n, g_{Euc})$.
\end{proof}

As mentioned above, in dimension $4$, there are examples of ALE spaces which have negative mass \cite{LeBrun, HeinLeBrun}. So we note the following corollary of Theorem \ref{orbifold_PMT}. 
\begin{corollary} 
\label{c:discont}
If $(X,g)$ is a $4$-dimensional Riemannian orbifold with negative mass at an orbifold point $q$, then the mass function $m : X \rightarrow \RR$ is necessarily discontinuous at $q$. 
\end{corollary}

\subsection{Odd-dimensional cases}
If the dimension $n$ is odd, we have the following. 
\begin{proposition}Let $(X,g)$ is a compact Riemannian orbifold with isolated singularities
and odd-dimensional. Then any nontrivial orbifold point must have $\Gamma = \ZZ/2\ZZ$.
Furthermore, $(X,g)$ is a good orbifold. That is, there is a $\ZZ/2\ZZ$ action 
on a compact manifold $\tilde{X}$ with finitely many fixed points 
such that $X = \tilde{X} / (\ZZ/2\ZZ)$. Letting $\pi : \tilde{X} \rightarrow X$ denote the quotient mapping, then $\pi^* g$ is a smooth Riemannian metric on $\tilde{X}$. 
\end{proposition}
\begin{proof}
For $n$ odd, any element $A \in \mathrm{O}(n)$ must have $\pm 1$ as an eigenvalue, so the only possibility for a nontrivial orbifold point is $\Gamma = \ZZ/2\ZZ$. 
Near any singular point $q$, a small distance sphere is homeomorphic to 
$\mathbb{RP}^{n-1}$, which is non-orientable if $n$ is odd. So if there is any nontrivial orbifold point, then $X$ contains a non-orientable $2$-sided hypersurface, which implies that $X \setminus \Sigma$ is non-orientable,  where $\Sigma$ is the finite set of singular points.  Let $\pi: X' \rightarrow X \setminus \Sigma$ denote the orientable double cover. 
Consider $\tilde{X} = X' \cup \{ \tilde{q}_1, \dots, \tilde{q}_j \}$ where $\tilde{q}_j$ are points, and extend $\pi: \tilde{X} \rightarrow X$ by letting $\pi(\tilde{q}_j) = q_j$. 
We extend to $\ZZ/2\ZZ$ action to $\tilde{X}$ with fixed points at $\tilde{q}_j$, and endow $\tilde{X}$ with the quotient topology. It is then straightforward to show that $\tilde{X}$ is a smooth manifold and $\pi^* g$ extends as a smooth Riemannian metric to $\tilde{X}$.
\end{proof}

\begin{corollary} Let $(X,g)$ be a compact Riemannian orbifold with isolated singularities and odd-dimensional. Then the mass function $m : X \rightarrow \RR$ satisfies $m > 0$ everywhere, unless $(X,g)$ is conformal to $(S^n, g_{round})$ or a ``football'' metric $S^n/ ( \ZZ/2\ZZ)$ with exactly $2$ singular points. 
\end{corollary}
\begin{proof} At any smooth point of $X$, the mass of the Green's function metric is positive by Theorem \ref{orbifold_PMT}. If the mass at a smooth point were zero, then the Green's function metric would be Euclidean space, which would imply that
  $(X,g)$ is conformal to $(S^n, g_{round})$. At a singular point $q$, by uniqueness of the Green's function metric, the Green's function metric at $q$ must be the $\ZZ/2\ZZ$ quotient of the Green's function of $(\tilde{X}, \pi^* g)$ at $\tilde{q}$. Since $\tilde{X}$ is a manifold, by the usual positive mass theorem, the Green's function metric upstairs must have non-negative mass, so the Green's function downstairs must also. If the mass at an orbifold point was $0$, then the Green's function metric upstairs would have to be Euclidean space. This implies that the Green's function metric downstairs is a $\ZZ/2\ZZ$-quotient of Euclidean space, which implies that  $(X,g)$ is conformal to a ``football'' metric $S^n/ ( \ZZ/2\ZZ)$.
\end{proof}

The above remarks show that the odd-dimensional case of the orbifold Yamabe problem is equivalent to the $\ZZ/2\ZZ$-equivariant Yamabe problem on the manifold $\tilde{X}$. 

\section{Compactness preliminaries}
\label{prelim} 
Remarkable work in analyzing the blow-up points of equation \eqref{yamabeequ} has been done in \cite{Marques}, \cite{Li-Zhang}, \cite{Li-Zhu}, \cite{KMS}. In this section, we are going to quote some of their definitions and local results on manifolds, which also appear to hold on orbifolds by modification of proofs.

In the following context, we will write $\RR^n/\Gamma$ or $B_r(\bar x)/\Gamma$. If $\Gamma=\{e\}$ is the trivial group, it denotes the Euclidean space or a smooth ball; if $\Gamma\neq \{e\}$ is a finite nontrivial group in $O(n)$, it denotes the Euclidean cone or a quotient of a ball centered at a singular point $\bar x$.

\subsection{Conformal scalar curvature equation}
\label{csce}
Instead of dealing with equation \eqref{yamabeequ}, we will study the following conformal scalar curvature equation. Let $\Omega\subset \RR^n/\Gamma$ be an open neighborhood of the origin, and suppose $g$ is a Riemannian metric in $\Omega$. Suppose also $f$ is a positive $C^1$ function defined in $\Omega$. Consider a positive $C^2$ function $u$ satisfying
\begin{align}
\label{conformalequ}
L_g u + K f^{-\tau}u^p = 0,
\end{align}
where $K$ is a positive $C^2$ function, $1< p\leq \frac{n+2}{n-2}$ and $\tau = \frac{n+2}{n-2} - p$.


We note that \eqref{conformalequ} is scale invariant. To see this, let $u$ be a solution to equation \eqref{conformalequ}. For any $s>0$, define the rescaled solution $v(y)=s^{\frac{2}{p-1}}u(sy)$. Then $L_hv + \tilde K\tilde f^{-\tau}v^p = 0$, where $\tilde K(y) = K(sy)$, $\tilde f(y)= f(sy)$ and the components in metric $h$ in normal coordinates are given by $h_{ij}(y) = g_{ij}(sy)$. Note that $v$ satisfies an equation of the same type as equation \eqref{conformalequ}.

We also note that \eqref{conformalequ} is conformally invariant. To see this, suppose $\tilde g = \phi^{\frac{4}{n-2}}g$ is a metric conformal to $g$ and let $u$ be a solution to equation \eqref{conformalequ}. Then $\phi^{-1}u$ satisfies $L_{\tilde g}(\phi^{-1}u)+ K(\phi f)^{-\tau}(\phi^{-1}u)^p = 0$, which is again an equation of the same type.

These two properties, will be used later to take a rescaling of the coordinates and conformally map $g$ to some conformal normal metric, without changing the type of equation \eqref{conformalequ}.

\subsection{Pohozaev Identity}
Suppose $u:B_\s(0)/\Gamma\to\RR$ is a positive $C^2$ solution to the equation
\begin{align}
a^{ij}(x)\p_{ij}u + b^i(x)\p_i u + c(x) u + K(x) u ^p = 0,
\end{align}
where $p\neq -1, K\in C^1$ and $a^{ij}, b^i, c$ are continuous functions, $1\leq i,j\leq n$. Define
\begin{align}
P(r, u) = \int_{\{|x| =r\}/\Gamma}\Big(\frac{n-2}{2}u\frac{\p u}{\p \s}-\frac{r}{2}|\nabla u|^2 + r\Big|\frac{\p u }{\p r}\Big|^2\Big)d\s(r),
\end{align}
for $0<r<\s$. Then, we have the following lemma.
\begin{lemma}[\cite{Marques} Lemma 2.1]
\label{poholemma} For all $0 < r <\s$,
\begin{align}
\begin{split}
P(r, u) = & -\int_{\{|x|\leq r\}/\Gamma}\Big(x\cdot \nabla u + \frac{n-2}{2}u\Big)((a^{ij}-\delta^{ij})\p_{ij}u + b^i\p_i u)dx\\
&+\int_{\{|x|\leq r\}/\Gamma}\Big(\frac{1}{2}x\cdot \nabla c + c\Big)u^2dx -\frac{r}{2}\int_{\{|x|=r\}/\Gamma}cu^2d\s(r)\\
&+\frac{1}{p+1}\int_{\{|x|\leq r\}/\Gamma}(x\cdot \nabla K(x))u^{p+1}dx \\
&+\Big(\frac{n}{p+1}-\frac{n-2}{2}\Big)\int_{\{|x|\leq r\}/\Gamma}K(x)u^{p+1}dx\\
&-\int_{\{|x| = r\}/\Gamma}\frac{1}{p+1}K(x)ru^{p+1}d\s(r).
\end{split}
\end{align}
\end{lemma}

\subsection{Isolated and isolated simple blow-up points}
Let $\Omega = B_\s(\bar x)/\Gamma$ be (a quotient of) an open ball centered at point $\bar x$. Suppose $\{g_k\}$ is a sequence of Riemannian metrics in $\Omega$ converging, in the $C^2_{loc}$ topology, to a metric $g$. Let $R_k$ denote the scalar curvature of $g_k$ and $R_g$ denote the scalar curvature of the limit metric $g$.
Suppose $\{f_k\}$ is a sequence of positive $C^1$ functions converging in the $C^1_{loc}$ topology to a positive function $f$. Also suppose $\{K_k\}$ is a sequence of positive $C^2$ functions converging in the $C^2_{loc}$ topology to a positive function $K_\infty$.
Consider a sequence of positive $C^2$ functions $u_k$ satisfying
\begin{align}
\label{confseq}
L_{g_k}u_k + K_kf_k^{-\tau_k}u_k^{p_k} = 0\ \mathrm{in}\ \Omega,
\end{align}
where $1+\epsilon_0 < p_k \leq \frac{n+2}{n-2}$ for some $\epsilon_0 > 0$ and $\tau_k = \frac{n+2}{n-2} - p_k$.

\begin{definition}
\label{iso_def}
Suppose $u_k$ is a sequence of positive functions satisfying equation \eqref{confseq}.
If $\Gamma =\{e\}$, define $\bar x$ to be an isolated blow-up point for $u_k$ if there exists a sequence $x_k\in \Omega$, converging to $\bar x$, so that:
1) $x_k$ is a local maximum point of $u_k$;
2) $M_k:=u_k(x_k)\to \infty$ as $k\to \infty$;
3) there exist $r, C > 0$ such that $u_k(x) \leq Cd_{g_k}(x, x_k)^{-\frac{2}{p_k-1}}$ for every $x\in B_r(x_k)\subset \Omega$, where $B_r(x_k)$ denotes the geodesic ball of radius $r$, centered at $x_k$, with respect to the metric $g_k$.

If $\Gamma\neq \{e\}$, let $\pi:B_\s(\bar x) \to B_\s(\bar x)/\Gamma$ be the projection map. We say that $\bar x$ is an isolated blow-up point for $u_k$ if $\pi^*(\bar x) = \bar x$ is an isolated blow-up point for $\pi^*(u_k)$ in the lifting-up space.
\end{definition}

Define
\begin{align}
\label{U_def}
U_c(y) = \Big(\frac{n(n-2)}{c}\Big)^{\frac{n-2}{4}}(1+|y|^2)^{\frac{2-n}{2}}
\end{align}
on $\RR^n/\Gamma$, where $c$ is some positive constant. It is not hard to check that
\begin{align}
\Delta U_c(y) + c U_c^{\frac{n+2}{n-2}}(y) = 0.
\end{align}

\begin{remark}
\label{coord_setting}
From now on, for each $k$, assume that we work in $g_k$-normal coordinates $\{x^i\}$ centered at point $x_k$. Then, we will simply write $u_k(x)$ instead of $u_k(\exp_{x_k}(x))$ and $|x|$ instead of $d_{g_k}(x,x_k)$.

Moreover, by \cite[Theorem 5.1]{Lee-Parker}, there exists a conformal factor $\phi_k$, such that $\tilde g_k = \phi_k^{\frac{4}{n-2}} g_k$ is the conformal normal metric with conformal normal coordinates $\{\tilde x^i\}$ centered at $x_k$. Let $\tilde u_k = \phi_k^{-1}u_k$. By the property that equation \eqref{conformalequ} is conformally invariant as stated in Section \ref{csce}, $u_k$ and $\tilde u_k$ satisfies the same type of conformal scalar curvature equation \eqref{conformalequ}. Hence we may assume $g_k$ is already the conformal normal metric and $\{x^i\}$ is already the conformal normal coordiantes centered at $x_k$.
\end{remark}

Consider the change of variables
\begin{align}
\label{cov}
y = M_k^{\frac{p_k-1}{2}}x,
\end{align}
and define the rescaled metric and functions
\begin{align}
\label{res_func}
\begin{split}
&(h_k)_{ij}(y) = (g_k)_{ij}(M_k^{-\frac{p_k-1}{2}}y),\ v_k(y) = M_k^{-1}u_k(M_k^{-\frac{p_k-1}{2}}y),\\
&\tilde f_k(y) = f_k(M_k^{-\frac{p_k-1}{2}}y),\ \tilde K_k(y) = K_k(M_k^{-\frac{p_k-1}{2}}y)\mathrm{\ and\ }\tilde R_k(y) = R_k(M_k^{-\frac{p_k-1}{2}}y),
\end{split}
\end{align}
for $|y| < M_k^{\frac{p_k-1}{2}}r$, where $r$ is as in Definition \ref{iso_def}.
Using the property that equation $\eqref{conformalequ}$ is rescale invariant as stated in Section \ref{csce}, the rescaled functions satisfy
\begin{align}
\label{res_equ}
L_{h_k}v_k + \tilde K_k\tilde f_k^{-\tau_k}v_k^{p_k} = 0.
\end{align}
The following property holds for an isolated blow-up point. 
\begin{proposition}[\cite{Marques} Proposition 4.3]
\label{p:m43}
Assume $u_k$ is a sequence of positive functions satisfying equation \eqref{confseq} and $x_k\to \bar x$ is an isolated blow-up point. Moreover, if $\Gamma\neq \{e\}$, we require $x_k = \bar x$ for all large $k$'s. Assume $p_k \to \frac{n+2}{n-2}$, then there exist 
$R'_k \to \infty$ and $\epsilon_k \to 0$, such that after passing to a subsequence, 
\begin{align}
\Big \Vert v_k(y) - U_{K_\infty(\bar{x})} (y) \Big \Vert_{C^2(B_{R'_k}(0))/\Gamma} \leq \epsilon_k,
\end{align}
and 
\begin{align}
\label{e:me2}
\frac{R'_k}{\log(M_k)} \to 0,
\end{align}
as $k \to \infty$.
\end{proposition}
\begin{proof} The case $K_k=constant$ is proved in \cite[Proposition 4.3]{Marques}.
For variable $K_k$, because $K_k\to K_\infty$ in the $C^0_{loc}$ norm, $\tilde K_k\to K_\infty(\bar x)$ in the $C^0_{loc}$ norm. By equation \eqref{res_equ} and the proof of \cite[Proposition 4.3]{Marques}, after passing to a subsequence, $v_k\to v>0$ in $C^2_{loc}$ norm, which satisfies
\begin{align}
\left\{\begin{array}{ll}\Delta v(y) + K_\infty(\bar x) v(y)^{p} = 0,\ \ y\in\RR^n/\Gamma\\v(0)=1,\ \ \nabla v(0)=0,\end{array}\right.
\end{align}
where $p = \lim_{k\to\infty}p_k$ and $\Delta$ denotes the Euclidean Laplacian. By \cite{G-S}, we must have $p=\frac{n+2}{n-2}$, and $v(y) = U_{K_\infty(\bar x)}(y)$. 
The proof of \eqref{e:me2} is the same as in \cite{Marques}.
\end{proof}

\begin{remark}
Under the same assumptions as Proposition \ref{p:m43}, we also have
\begin{align}
\Big \Vert v_k(y) - U_{K_k(x_k)} (y) \Big \Vert_{C^2(B_{R'_k}(0))/\Gamma} \leq \epsilon_k,
\end{align}
which is simply because $U_{K_k(x_k)}\to U_{K_\infty(\bar{x})}$ uniformly in $\RR^n/\Gamma$.
\end{remark}

\begin{definition}
Suppose $u_k$ is a sequence of positive functions satisfying equation \eqref{confseq} and $x_k\to \bar x$ is an isolated blow-up point. If $\Gamma =\{e\}$, define
\begin{align}
\bar u_k(r) = \frac{1}{Vol(S^{n-1})r^{n-1}}\int_{\p B_r(x_k)}u_kd\s(r),
\end{align}
where we are using $g_k$-normal coordinates and integrating with respect to the Euclidean volume form. We say $x_k\to\bar x$ is an isolated simple blow-up point if there exists a real number $0<\rho<r$ such that the functions
\begin{align}
\hat u_k(r) \equiv r^{\frac{2}{p_k-1}}\bar u_k(r)
\end{align}
have exactly one critical point in the interval $(0,\rho)$, for $k$ large.

If $\Gamma\neq \{e\}$, let $\pi:B_\s(\bar x)\to B_\s(\bar x)/\Gamma$ be the projection map. We say that $\bar x$ is an isolated simple blow-up point for $u_k$ if $\pi^*(\bar x)$ is an isolated simple blow-up point for $\pi^*(u_k)$ in the lifting-up space.
\end{definition}

Next, we give the special blow-up condition that we will work on.
\begin{condition}
\label{blow-up_cond}
Assume $u_k$ is a sequence of positive functions satisfying equation \eqref{confseq} and $x_k\to \bar x$ is an isolated simple blow-up point. 
Moreover, if  $\Gamma_{\bar{x}} \neq \{e\}$, then we require that $x_k = \bar x$ for all $k$ sufficiently large.
\end{condition}

\begin{remark}
In Condition \ref{blow-up_cond}, the main reason why we require $x_k = \bar x$ for all $k$ sufficiently large in the singular point case $\Gamma\neq\{e\}$ is the following. 
By assuming so, for each $k$, the geodesic ball centered at $x_k$ will be $B_r(\bar x)/\Gamma$. Then, when we later analyze some local integrals and let $k\to \infty$, we will not run into the case that integrals over smooth balls converge to an integral over a quotient of a smooth ball.  We also note that eventually, we will prove in Corollary~\ref{blow_is_iso} that if blow-up occurs at a singular point, then  Condition \ref{blow-up_cond} must hold.
\end{remark}

From now on, assume we are working in dimension $n=4$. In the following context, we will use $C$ to denote various constants which only depend on the limit metric $g$, $\inf K_\infty$, $\Vert K_\infty\Vert_{C^2}$ and possibly the chosen small radius $\rho_1$, $\delta$ and $\s$. The dependency is implied in the proof. Fix $\delta > 0$, and define 
\begin{align}
\lambda_k = (2 - \delta) \frac{p_k-1}{2} - 1.
\end{align}
\begin{proposition}[\cite{Li-Zhu}]
\label{p:lz73}
Assuming Condition \ref{blow-up_cond}, for sufficiently small $\delta > 0$, there exists 
constants $0 < \rho_1 < \rho$ and $C > 0$ such that 
\begin{align}
\begin{split}
M_k^{\lambda_k} u_k(x) \leq C|x|^{-2 + \delta},\\
M_k^{\lambda_k} |\nabla u_k(x)| \leq C |x|^{-3 + \delta},\\
M_k^{\lambda_k} |\nabla^2 u_k(x)| \leq C |x|^{-4 + \delta},
\end{split}
\end{align}
for every $x$ satisfying  
\begin{align}
R'_k M_k^{-\frac{p_k-1}{2}} \leq |x| \leq \rho_1.
\end{align}
As a consequence, it implies
\begin{align}
\begin{split}
v_k(y)\leq CM_k^{\delta\frac{p_k-1}{2}}(1+|y|)^{-2},\\
|\nabla v_k(y)|\leq CM_k^{\delta\frac{p_k-1}{2}}(1+|y|)^{-3},\\
|\nabla^2v_k(y)|\leq CM_k^{\delta\frac{p_k-1}{2}}(1+|y|)^{-4},
\end{split}
\end{align}
for every $y$ satisfying
\begin{align}
|y|\leq \rho_1M_k^{\frac{p_k-1}{2}}.
\end{align}
\end{proposition}
\begin{proof} The proof is the same as \cite[Lemma 7.3]{Li-Zhu}. That proof was for $n =3$, but directly generalizes to higher dimensions. 
\end{proof}

\begin{proposition}
\label{grad0}
Assuming Condition \ref{blow-up_cond}, then 
there exists a constant $C$ such that
\begin{align}
|\nabla K_k( x_k )| \leq C  M_k^{-2+2\delta},
\end{align}
as $k \to \infty$. Consequently, 
\begin{align}
|\nabla K_{\infty} ( \bar{x})| = 0.
\end{align}
\end{proposition}
\begin{proof} The proof is very similar to \cite[Lemma 7.8]{Li-Zhu}, we provide here only a brief outline. Recall that $\{x^i\}$ is the $g_k$-normal coordinates centered at point $x_k$. For some fixed positive small $\s$, let $\eta$ be a smooth cutoff function such that $\eta(x) = 1$ for $|x|\leq \s/2$ and $\eta(x) =0$ for $|x|\geq \s$. Multipying equation \eqref{confseq} by $\eta(\p u_k/\p x^j)$, integrating by parts on $\{|x|\leq \s\}/\Gamma$, we get
\begin{align}
\begin{split}
&\int_{\{|x|\leq \s\}/\Gamma}\frac{\p K_k}{\p x^j} \eta f_k^{-\tau_k}u_k^{p_k+1}dx\\
\leq &C\int_{\{|x|\leq \s\}/\Gamma}\Big(|\nabla \eta|\cdot |\nabla u_k|^2 + \Big|\frac{\p(\eta R_{k})}{\p x^j}\Big|u_k^2 + \Big|\frac{\p( \eta  f_k^{-\tau_k})}{\p x^j}\Big|K_ku_k^{p_k+1}\Big)dx.
\end{split}
\end{align}
Then, using Proposition \ref{p:m43} in the ball $|x|\leq R'_kM_k^{-\frac{p_k-1}{2}}$ and Proposition \ref{p:lz73} in the annuli $R'_kM_k^{-\frac{p_k-1}{2}}\leq |x|\leq \s$, together with the assumption that $K_k$ converges to $K_\infty$ in the $C^0_{loc}$ norm, for large $k$'s we have
\begin{align}
\label{partial_K}
\int_{\{|x|\leq \s\}/\Gamma}\frac{\p K_k}{\p x^j}u_k^{p_k+1}dx\leq C_1 M_k^{-2+2\delta},
\end{align}
where $C_1$ depends on $\inf K_\infty$, $\Vert K_\infty\Vert_{C^0}$ and $\s$.
Next, the power series expansion
\begin{align}
\frac{\p K_k}{\p x^j}(x) = \frac{\p K_k}{\p x^j}(0) + \frac{\p^2 K_k}{\p x^j\p x^i}(0)\cdot x^i + O(|x|^2),
\end{align}
implies that
\begin{align}
\Big|\frac{\p K_k}{\p x^j}(0)\Big| \leq \Big|\frac{\p K_k}{\p x^j}(x)\Big| + C_2|x|,
\end{align}
where $C_2$ depends on $\Vert K_\infty\Vert_{C^2}$. Multiplying  by $u_k^{p_k+1}$, integrating over $\{|x|\leq \s\}/\Gamma$ and using inequality \eqref{partial_K}, we get
\begin{align}
\Big|\frac{\p K_k}{\p x^j}(0)\Big|\int_{\{|x|\leq \s\}/\Gamma} u_k(x)^{p_k+1}dx \leq C\Big(M_k^{-2+2\delta} +\int_{\{|x|\leq \s\}/\Gamma} |x|u_k(x)^{p_k+1}dx\Big).
\end{align}
The integral on the left limits to the volume of the bubble which is a finite constant, the integral on the right can be estimated similarly as above using Proposition \ref{p:m43} and Proposition \ref{p:lz73}. Therefore, we have proved
\begin{align}
|\nabla K_k( x_k )| \leq C M_k^{-2+2\delta},
\end{align}
where $C$ is a constant depending on $\inf K_\infty$, $\Vert K_\infty \Vert_{C^2}$, $g$ and $\s$.
\end{proof}

\begin{proposition}[\cite{Marques} Proposition 4.5] 
  Assuming Condition \ref{blow-up_cond}, then there exists a constant $C > 0$ 
and $0 < \rho_1 < \rho$ such that 
\begin{align}
M_k u_k (x) \leq C d_{g_k}(x,x_k)^{-2}
\end{align}
for $x$ satisfying $d_{g_k}(x,x_k) \leq \rho_1$.
\end{proposition}
\begin{proof}  The proof is very similar to the proof of \cite[Proposition 4.5]{Marques}. That proof was assuming $K_k = constant$.   For variable $K_k$, every step in the proof remains valid, except for Claim~2 of \cite[Proposition 4.5]{Marques}, which says that there exists $C > 0$ such that 
\begin{align}
\tau_k \leq C M_k^{-2 + 2 \delta + o(1)}log(M_k)\ \ \mathrm{as\ \ }k\to\infty.
\end{align}
This estimate does not hold in our setting, however, a modification of his arguments shows that there exists a constant $C > 0$ such that
\begin{align}
\label{tau_weak_est}
\tau_k \leq C M_k^{-2 + 6 \delta + o(1)}\ \ \mathrm{as\ \ }k\to\infty.
\end{align}
To verify this, note that when $K_k$ is a variable function, there will be an extra term on the left hand side of \cite[inequality (4.18)]{Marques}. That extra term is
\begin{align}
\begin{split}
&\frac{1}{p_k+1}\Big|\int_{\{|x|\leq \frac{\rho_1}{2}\}/\Gamma}(\nabla K_k(x)\cdot x)u_k(x)^{p_k+1}dx\Big|\\
\leq &\frac{1}{p_k+1}M_k^{3-p_k}\int_{\{|y|\leq \frac{\rho_1}{2}M_k^{\frac{p_k-1}{2}}\}/\Gamma}|\nabla_x K_k(M_k^{-\frac{p_k-1}{2}}y)|\cdot |M_k^{-\frac{p_k-1}{2}}y|v_k(y)^{p_k+1}dy.
\end{split}
\end{align}
For small $\rho_1$ and large $k$, when $|x|\leq \frac{\rho_1}{2}$, by power series expansion and Proposition~\ref{grad0}, we have
\begin{align}
|\nabla K_k(x)| \leq |\nabla K_k(x_k)| + C|x| \leq C(M_k^{-2+2\delta} + |x|).
\end{align}
Then for $|y|\leq \frac{\rho_1}{2}M_k^{\frac{p_k-1}{2}}$ and large $k$,
\begin{align}
\begin{split}
|\nabla_x K_k(M_k^{-\frac{p_k-1}{2}}y)|\cdot |M_k^{-\frac{p_k-1}{2}}y| \leq &C(M_k^{-2-\frac{p_k-1}{2}+2\delta}|y| + M_k^{-p_k+1}|y|^2)\\
\leq& C(\rho_1M_k^{-2+2\delta} + M_k^{-p_k+1})|y|^2\\
\leq & C M_k^{-2+2\delta}|y|^2.
\end{split}
\end{align}
On the other hand, by Proposition \ref{p:lz73}, we know
\begin{align}
v_k(y)^{p_k+1}\leq C M_k^{\delta\frac{p_k^2-1}{2}}(1+|y|)^{-2(p_k+1)}.
\end{align}
Therefore,
\begin{align}
\begin{split}
&M_k^{3-p_k}\int_{\{|y|\leq \frac{\rho_1}{2}M_k^{\frac{p_k-1}{2}}\}/\Gamma}|\nabla_x K_k(M_k^{-\frac{p_k-1}{2}}y)|\cdot |M_k^{-\frac{p_k-1}{2}}y|v_k(y)^{p_k+1}dy\\
\leq &CM_k^{3-p_k}M_k^{-2+2\delta}M_k^{\delta\frac{p_k^2-1}{2}}\int_{\{|y|\leq \frac{\rho_1}{2}M_k^{\frac{p_k-1}{2}}\}/\Gamma}(1+|y|)^{-2(p_k+1)}|y|^2dy\\
\leq &CM_k^{-2+6\delta+o(1)}\int_{\RR^4/\Gamma}(1+|y|)^{-8+o(1)}|y|^2dy\\
\leq& C M_k^{-2+6\delta+o(1)}\ \ \mathrm{as\ }k\to\infty.
\end{split}
\end{align}
Then \eqref{tau_weak_est} is proved following the same proof as in Claim 2 in \cite[Proposition 4.5]{Marques}.
\end{proof}

\begin{corollary}[\cite{Marques} Corollary 4.6]
\label{greenlimit}
Assuming Condition \ref{blow-up_cond}, after maybe passing to a subsequence, we have
\begin{align}
M_ku_k\to h_{\bar x}\ \mathrm{in}\ C^2_{loc}((B_r(\bar x)\setminus\{\bar x\})/\Gamma),
\end{align}
where $M_k$ is as defined in Definition \ref{iso_def} and $h_{\bar x} = aG(\cdot, \bar x)$ is a constant multiple of the standard Green function, i.e. $L_g(G(\cdot, \bar x)) = \delta_{\bar x}$ is the Dirac delta function at point $\bar x$. (Here, $g$ stands for the limit metric.)
\end{corollary}
\begin{proof}
The proof of Marques remains valid for variable $K_k$. 
\end{proof}
Then, we have the following:
\begin{proposition}\label{vUdiff}
Assuming Condition \ref{blow-up_cond}, then
\begin{align}
\label{tauest}
\tau_k\leq C M_k^{-2},
\end{align}
and there exists $\delta > 0$ such that
\begin{align}
\label{v_bound}
|v_k(y) - U_{K_{k}(x_k)}(y)| &\leq C M_k^{-2},\\
\label{v'_bound}
|\nabla(v_k-U_{K_{k}(x_k)})(y)|&\leq C M_k^{-2}(1+|y|)^{-1},\\
\label{v''_bound}
|\nabla^2(v_k-U_{K_{k}(x_k)})(y)|&\leq C M_k^{-2}(1+|y|)^{-2},
\end{align}
for $|y|\leq \delta M_i^{\frac{p_i-1}{2}}$.
\end{proposition}
\begin{proof}
The proof for $K_k=constant$ is by \cite[Chapter 5]{Marques}. We need to verify that Marques' proof is valid for variable $K_k$. When $K_k$ is a variable function instead of a fixed constant, every $U_0$ in \cite[Chapter 5]{Marques} has to be replaced by $U_{K_{k}(x_k)}$, then \cite[equation (5.1)]{Marques} will become
\begin{align}
\notag
Q_k(y)
=&\Lambda_k^{-1} \Big\{ c(n)M_k^{-(p_k-1)}R_{g_k}\Big(M_k^{-\frac{p_k-1}{2}}y\Big)U_{K_{k}(x_k)}(y) + M_k^{-(1+N)\frac{p_k-1}{2}}O(|y|^N)|y|(1+|y|^2)^{-2}\\
\label{diff_Q_term}
&+\Big(K_k(x_k)U_{K_{k}(x_k)}^3 - \tilde K_k(y)\tilde f_k^{-\tau_k}U^{p_k}_{K_{k}(x_k)}\Big)\Big\},
\end{align}
where $\tilde K_k(y) = K_k(M_k^{-\frac{p_k-1}{2}}y)$ and $\tilde f_k(y) = f_k(M_k^{-\frac{p_k-1}{2}}y)$. Only the last term in \eqref{diff_Q_term} differs from the last term in \cite[equation (5.1)]{Marques}. Denote it by
\begin{align}
\mathbb{B}_k = K_k(x_k)U_{K_{k}(x_k)}^3 - \tilde K_k(y)\tilde f_k^{-\tau_k}U^{p_k}_{K_{k}(x_k)}.
\end{align}
To analyze it, by power series expansion we obtain
\begin{align}
\begin{split}
|\tilde K_k(y) - K_k(x_k)| =& |K_k(M_k^{-\frac{p_k-1}{2}}y) - K_k(0)|\\
\leq & |\nabla K_k(0)|M_k^{-\frac{p_k-1}{2}}|y| + C M_k^{-(p_k-1)}|y|^2.
\end{split}
\end{align}
By Proposition \ref{p:lz73}, $|\nabla K_k(0)| =  |\nabla K_k(x_k)|\leq CM_k^{-2+2\delta}$, hence
\begin{align}
\begin{split}
|\tilde K_k(y) - K_k(x_k)|\leq & C\Big(M_k^{-2-\frac{p_k-1}{2}+2\delta}|y| + M_k^{-(p_k-1)}|y|^2\Big)\\
\leq& C \Big(M_k^{-3+2\delta+o(1)} + M_k^{-2+\tau_k}|y|\Big)|y|.
\end{split}
\end{align}
Then we have the estimate
\begin{align}
\begin{split}
\mathbb{B}_k\leq&K_k(x_k)U_{K_{k}(x_k)}^{3}(1-(\tilde f_kU_{K_{k}(x_k)}^{3})^{-\tau_k}) + |\tilde K_k(y) - K_k(x_k)|\tilde f_k^{-\tau_k}U^{p_k}_{K_{k}(x_k)}\\
\leq& C \Big(\tau_k\log(\tilde f_kU_{K_{k}(x_k)})(1+|y|^2)^{-3} + (M_k^{-3+2\delta+o(1)} + M_k^{-2+\tau_k}|y|)|y|(1+|y|^2)^{-2}\Big).
\end{split}
\end{align}
Suppose
\begin{align}
\Lambda_k^{-1}\max\{M_k^{-2},\tau_k\}\to 0\ \ \mathrm{as\ }k\to\infty.
\end{align}
Note that \eqref{tau_weak_est} implies $\lim_{k\to\infty}M_k^{\tau_k} = 1$, so $\lim_{k\to\infty}\Lambda_k^{-1}\mathbb{B}_k\to 0$. Thus $|Q_k|\to 0$ as $k\to 0$ and \cite[Lemma 5.1]{Marques} can be proved following the rest of his proof. Also, \cite[Lemma 5.3]{Marques} remains valid with a similar modification to the term $\tilde Q_k$ in his proof. Therefore, our proposition is proved for variable $K_k$.
\end{proof}

\section{Local blow-up analysis}\label{local_ana}

\subsection{Application of Pohozaev Identity}
\label{app_poho}
Assuming Condition \ref{blow-up_cond}, recall that $g_k$ is the conformal normal metric with conformal normal coordinates $\{x^i\}$ centered at the point $x_k$. By \cite[Chapter~5]{Lee-Parker} and the explicit computations in \cite[Chapter 2]{Li-Zhang}, we have the following for each $k$. At $x = 0$ (the point $x_k$), the scalar curvature satisfies
\begin{align}
\label{Rest}
R_{k}(0) = 0, (R_{k})_{,i}(0) = 0.
\end{align}
Locally in a neighborhood of $x=0$, we have $(g_k)_{ij}= \delta_{ij} + O(r^2)$ and $\det(g_k) = 1 + O(r^N)$ for some $N \geq 5$. Write
\begin{align}
\label{laplacian}
\Delta_{g_k} = \frac{1}{\sqrt{\det(g_k)}}\p_i(\sqrt{\det(g_k)} (g_k)^{ij}\p_j) = \Delta_0 + (b_k)_i\p_i + (d_k)_{ij}\p_{ij},
\end{align}
where $\p_i=\frac{\p}{\p x^i}$, $\p_{ij}=\frac{\p^2}{\p x^i\p x^j}$ and $\Delta_0 = \sum_{i=1}^n\frac{\p^2}{(\p x^i)^2}$, then
\begin{align}
(b_k)_i = O(r^2)\ \ \mathrm{and\ \ }(d_k)_{ij}=g^{ij} - \delta_{ij}=O(r^2).
\end{align}
By the change of variables and the rescaled metric defined in \eqref{cov} and \eqref{res_func}, we have
\begin{align}
\label{rescaleh}
\Delta_{h_k} = \bar\Delta_0 + (\bar b_k)_i\bar\p_i + (\bar d_k)_{ij}\bar\p_{ij},
\end{align}
where
\begin{align}
\begin{split}
&(\bar b_k)_i(y) = M_k^{-\frac{p_k-1}{2}}(b_k)_i(M_k^{-\frac{p_k-1}{2}}y),\ \ (\bar d_k)_{ij}(y) = (d_k)_{ij}(M_k^{-\frac{p_k-1}{2}}y),\\
&\bar\p_i=\frac{\p}{\p y^i},\ \ \bar\p_{ij}=\frac{\p^2}{\p y^i\p y^j}\mathrm{\ \ and\ \ }\bar\Delta_0 = \sum_{i=1}^n\frac{\p^2}{(\p y^i)^2}.
\end{split}
\end{align}
It follows that
\begin{align}
\label{bdest}
|(\bar b_k)_i(y)| = M_k^{-\frac{3(p_k-1)}{2}}O(|y|)^2,\ \ |(\bar d_k)_{ij}(y)| = M_k^{-(p_k-1)}O(|y|)^2.
\end{align}
Using the above notation, we rewrite equation \eqref{confseq} in the $\{x^i\}$ coordinates as following
\begin{align}
((d_k)_{ij} + \delta_{ij})\p_{ij}u_k + (b_k)_i\p_i u_k - \frac{1}{6}R_ku_k + K_kf_k^{-\tau_k}u_k^{p_k}=0.
\end{align}
For some small $\s$, apply Lemma \ref{poholemma} to the above equation in $\{|x|\leq \s\}/\Gamma$, to obtain
\begin{align}
\label{pohoapply}
\begin{split}
P(\s, u_k) = &-\int_{\{|x|\leq \s\}/\Gamma}\Big(x\cdot \nabla_x u_k + u_k\Big)((d_k)_{ij}\p_{ij}u_k + (b_k)_i\p_i u_k) \\
&-\frac{1}{12}\int_{\{|x|\leq \s\}/\Gamma}(x\cdot \nabla_x R_k + 2R_k)u_k^2\\
&+\frac{\s}{12}\int_{\{|x|=\s\}/\Gamma}R_ku_k^2 + \frac{1}{p_k+1}\int_{\{|x|\leq \s\}/\Gamma}(x\cdot \nabla_x(K_kf_k^{-\tau_k}))u_k^{p_k+1}\\
&+\frac{\tau_k}{p_k+1}\int_{\{|x|\leq \s\}/\Gamma}K_kf_k^{-\tau_k}u_k^{p_k+1}-\frac{\s}{p_k+1}\int_{\{|x|=\s\}/\Gamma}K_kf_k^{-\tau_k}u_k^{p_k+1}\\
=&\int_{\{|x| =\s\}/\Gamma}\Big(u_k\frac{\p u_k}{\p \nu}-\frac{\s}{2}|\nabla u_k|^2 + \s\Big|\frac{\p u_k }{\p \nu}\Big|^2\Big).
\end{split}
\end{align}
Denote $R_k' = M_k^{\frac{p_k-1}{2}}\s$. Define the following
\begin{align}
\begin{split}
I_{k,1} =&-\frac{\s}{p_k+1}\int_{\{|x|=\s\}/\Gamma}K_kf_k^{-\tau_k}u_k^{p_k+1}d\s(x)\\
=&-\frac{\s}{p_k+1}M_k^{\frac{5-p_k}{2}}\int_{\{|y|=R_k'\}/\Gamma}\tilde K_k\tilde f_k^{-\tau_k}v_k^{p_k+1}d\s(y),
\end{split}
\end{align}
\begin{align}
\begin{split}
I_{k,2} =&\int_{\{|x|\leq \s\}/\Gamma}(-(b_k)_i\p_iu_k-(d_k)_{ij}\p_{ij}u_k)(\n_x u_k\cdot x + u_k)dx\\
=&M_k^{\tau_k}\int_{\{|y|\leq R_k'\}/\Gamma}(-(\bar b_k)_i\bar\p_i v_k-(\bar d_k)_{ij}\bar\p_{ij}v_k)(y\cdot \n_y v_k + v_k)dy,
\end{split}
\end{align}
\begin{align}
\begin{split}
I_{k,3} =&\frac{\s}{12}\int_{\{|x|=\s\}/\Gamma}R_ku_k^2d\s(x)
= \frac{\s}{12}M_k^{\frac{7-3p_k}{2}}\int_{\{|y|=R_k'\}/\Gamma}\tilde R_kv_k^2d\s(y),
\end{split}
\end{align}
\begin{align}
\begin{split}
I_{k,4} =&-\frac{1}{12}\int_{\{|x|\leq \s\}/\Gamma}(x\cdot\n_x R_k+2R_k)u_k^2dx\\
&= -\frac{1}{12}M_k^{4-2p_k}\int_{\{|y|\leq R_k'\}/\Gamma}(y\cdot\n_y \tilde R_k+2\tilde R_k)v_k^2dy,
\end{split}
\end{align}
\begin{align}
\begin{split}
I_{k,5} =& \frac{1}{p_k+1}\int_{\{|x|\leq \s\}/\Gamma}(x\cdot \nabla_x(K_kf_k^{-\tau_k}))u_k^{p_k+1}dx\\
&=\frac{1}{p_k+1}M_k^{\tau_k}\int_{\{|y|\leq R_k'\}/\Gamma}(y\cdot \nabla_{y}(\tilde K_k\tilde f_k^{-\tau_k}))v_k^{p_k+1}dy,
\end{split}
\end{align}
\begin{align}
\begin{split}
I_{k,6} = &\frac{\tau_k}{p_k+1}\int_{\{|x|\leq \s\}/\Gamma}K_kf_k^{-\tau_k}u_k^{p_k+1}dx=\frac{\tau_k}{p_k+1}M_k^{\tau_k}\int_{\{|y|\leq R_k'\}/\Gamma}\tilde K_k\tilde f_k^{-\tau_k}v_k^{p_k+1}dy,
\end{split}
\end{align}
\begin{align}
\begin{split}
I_{k,7}=&\int_{\{|x|=\s\}/\Gamma} \Big\{\Big(\Big|\frac{\p u_k}{\p \nu_x}\Big|^2 -\frac{1}{2}|\n_x u_k|^2\Big)\s+u_k\frac{\p u_k}{\p\nu_x}\Big\}d\s(x)\\
=&M_k^{\tau_k}\int_{\{|y|=R_k'\}/\Gamma}\Big\{\Big(\Big|\frac{\p v_k}{\p \nu_y}\Big|^2-\frac{1}{2}|\n_yv_k|^2\Big)R_k'+v_k\frac{\p v_k}{\p \nu_y}\Big\}d\s(y).
\end{split}
\end{align}
Then \eqref{pohoapply} becomes
\begin{align}
\label{pohozaev}
I_{k,7} = I_{k,1} + I_{k,2}+ I_{k,3}+ I_{k,4}+ I_{k,5}+ I_{k,6}.
\end{align}

\subsection{Main Estimates}\label{main_est_sect}
By \cite{G-S}, it is sufficient to consider the blow-up case as $p_k\to 3$, consequently $\tau_k=3-p_k\to 0$. Then we know $\lim_{k\to\infty}M_k^{\tau_k}=1$ from \eqref{tau_weak_est}. We can estimate terms in the Pohozaev identity through the following lemmas. We will use $C$, $C_1$ to denote various positive constants independent of $k$ and $\s$. For notational simplicity, we will omit $dx,dy,d\s(x)$ and $d\s(y)$ terms in integrals.
\begin{lemma}\label{lemma_i1} For small $\s > 0$,
\begin{align}
\label{I_1est}
\lim_{k\to\infty} M_k^2I_{k,1} = 0.
\end{align}
\end{lemma}
\begin{proof}
Using \eqref{confseq},
\begin{align}
I_{k,1} =\frac{\s}{p_k+1}\int_{\{|x|=\s\}/\Gamma}(L_{g_k}u_k)\cdot u_k.
\end{align}
By Corollary \ref{greenlimit}, $M_ku_k\to h$ in the $C^2_{loc}$ norm, hence for small $\s$,
\begin{align}
\begin{split}
\lim_{k\to\infty} M_k^2I_{k,1} =& \lim_{k\to\infty}\frac{\s}{p_k+1}\int_{\{|x|=\s\}/\Gamma}(L_{g_k}(M_ku_k))\cdot (M_ku_k)\\
&= \frac{\s}{p_k+1}\int_{\{|x|=\s\}/\Gamma}(L_{g}h)\cdot h =0.
\end{split}
\end{align}
\end{proof}

\begin{lemma}\label{lemma_i2}
For $\s \leq 1$, there exists some constant $C>0$ such that
\begin{align}
\label{I_2est}
\limsup_{k\to\infty}M_k^2|I_{k,2}| \leq C \s^2.
\end{align}
\end{lemma}
\begin{proof}
Because $U_{K_k(x_k)}(y)$ is a radial function and $\{y^i\}$ is a rescale of the conformal normal coordinates, we know
\begin{align}
(\Delta_{h_k} - \bar\Delta_0)U_{K_k(x_k)} = ((\bar b_k)_i\bar\p_i+(\bar d_k)_{ij}\bar\p_{ij})U_{K_k(x_k)}\equiv 0.
\end{align}
Then, we have
\begin{align}
\begin{split}
M_k^2|I_{k,2}| \leq& M_k^{2+\tau_k}\int_{\{|y|\leq \s M_k^{\frac{p_k-1}{2}}\}/\Gamma}|((\bar b_k)_i\p_i+(\bar d_k)_{ij}\p_{ij})v_k|\cdot|\n_y v_k\cdot y + v_k|\\
=&M_k^{2+\tau_k}\int_{\{|y|\leq \s M_k^{\frac{p_k-1}{2}}\}/\Gamma}|((\bar b_k)_i\p_i+(\bar d_k)_{ij}\p_{ij})(v_k-U_{K_k(x_k)})|\cdot|\n_y v_k\cdot y + v_k|.
\end{split}
\end{align}
For $\s\leq 1$, we have $|y|\leq \s M_k^{\frac{p_k-1}{2}}\leq M_k^{\frac{p_k-1}{2}}$, hence $M_k^{-1}\leq M_k^{-\frac{p_k-1}{2}}\leq C(1+|y|)^{-1}$ for large $k$. By \eqref{bdest} and Proposition \ref{vUdiff}, for large $k$, we have
\begin{align}
\begin{split}
&|((\bar b_k)_i(y)\bar \p_i+(\bar d_k)_{ij}(y)\bar \p_{ij})(v_k(y)-U_{K_k(x_k)}(y))|\\
\leq& CM_k^{-2}M_k^{-(p_k-1)}|y|^2(1+|y|)^{-1}(M_k^{-\frac{p_k-1}{2}}+(1+|y|)^{-1})\\
\leq& C M_k^{-1-p_k}|y|^2(1+|y|)^{-2} \leq  C M_k^{-1-p_k},
\end{split}
\end{align}
and
\begin{align}
\begin{split}
|y\cdot \n_y v_k(y) + v_k(y)| \leq& |\n_y U_{K_k(x_k)}|\cdot |y| + |\n_y(v_k-U)|\cdot |y| + |v_k-U| + U\\
\leq& C \Big(|y|^2(1+|y|^2)^{-2}+M_k^{-2}(1+|y|)^{-1}|y| + M_k^{-2}+(1+|y|^2)^{-1}\Big)\\
\leq&C \Big((1+|y|^2)^{-1} + M_k^{-2}\Big) \leq C (1+|y|^2)^{-1}.
\end{split}
\end{align}
Thus, for large $k$,
\begin{align}
\begin{split}
M_k^2|I_{k,2}| \leq& C M_k^{2+\tau_k}M_k^{-1-p_k}\int_{\{|y|\leq \s M_k^{\frac{p_k-1}{2}}\}/\Gamma}(1+|y|^2)^{-1}\\
\leq & C M_k^{1-p_k+\tau_k}\int_{0}^{\s M_k^{\frac{p_k-1}{2}}}r^3(1+r^2)^{-1}dr\\
\leq & C M_k^{1-p_k+\tau_k}((\s M_k^{\frac{p_k-1}{2}})^2 + C_1)\\
=& C M_k^{\tau_k}\s^2 + CC_1M_k^{1-p_k+\tau_k} \leq C \s^2\ \ \mathrm{as\ }k\to\infty.
\end{split}
\end{align}
\end{proof}
\begin{lemma}\label{lemma_i3}
There exist constants $C>0$ and $0< \delta < 1$ such that when $\s <\delta$,
\begin{align}
\label{I_3est}
\limsup_{k\to\infty}M_k^2|I_{k,3}| \leq C \s^2.
\end{align}
\end{lemma}
\begin{proof}
Due to \eqref{Rest}, there exist $C>0$ and $0<\delta <1$ such that when $\s <\delta$, on $|x|=\s$, by power series expansion, $|R_k(x)|\leq C|x|^2 = C\s^2$.
Hence $|\tilde R_k(y)| \leq C\s^2$ on $|y| = \s M_k^{\frac{p_k-1}{2}}$.
On the other hand, on $|y| = \s M_k^{\frac{p_k-1}{2}} \leq M_k^{\frac{p_k-1}{2}}$, $M_k^{-1}\leq M_k^{-\frac{p_k-1}{2}}\leq |y|^{-1}$. By Proposition \ref{vUdiff},
\begin{align}
\label{v_k^2}
\begin{split}
v_k^2(y) \leq& [|v_k(y)-U_{K_k(x_k)}(y)| + U_{K_k(x_k)}(y)]^2\\
\leq& C[M_k^{-2} + (1+|y|^2)^{-1}]^2 \leq C |y|^{-4}.
\end{split}
\end{align}
Then, we have the estimate
\begin{align}
\begin{split}
M_k^2|I_{k,3}|  \leq&\frac{\s}{12}M_k^{\frac{11-3p_k}{2}}\int_{\{|y|=\s M_k^{\frac{p_k-1}{2}}\}/\Gamma}|\tilde R_k|v_k^2\\
\leq &C\s^3M_k^{\frac{11-3p_k}{2}}\int_{\{|y|=\s M_k^{\frac{p_k-1}{2}}\}/\Gamma}|y|^{-4}\\
\leq &C \s^3M_k^{\frac{11-3p_k}{2}}(\s M_k^{\frac{p_k-1}{2}})^{-1}
= C \s^2 M_k^{2\tau_k} \leq C \s^2\ \ \mathrm{as\ }k\to\infty.
\end{split}
\end{align}
\end{proof}
\begin{lemma}\label{lemma_i4}
There exist constants $C>0$ and $0< \delta < 1$ such that when $\s <\delta$,
\begin{align}
\label{I_4est}
\limsup_{k\to\infty}M_k^2|I_{k,4}| \leq C \s^2.
\end{align}
\end{lemma}
\begin{proof}
By \eqref{Rest}, there exist $C>0$ and $0<\delta<1$ such that when $\s<\delta$, in the ball $|x|\leq\s$, by power series expansion, we have
\begin{align}
|R_k(x)|\leq C|x|^2, \quad |\n_x R_{k}(x)| \leq |\n_x R_{k}(0)| + C|x|\leq C|x|.
\end{align}
The second inequality implies
\begin{align}
|x\cdot \n_x R_{k}(x)|\leq |x|\cdot |\n_x R_{k}(x)|\leq C|x|^2.
\end{align}
It follows that in the ball $|y|\leq \s M_k^{\frac{p_k-1}{2}}$,
$|\tilde R_k(y)|\leq C M_k^{-(p_k-1)}|y|^2$, and
\begin{align}
|y\cdot \n_y \tilde R_{k}(y)| = |(M_k^{-\frac{p_k-1}{2}}y)\cdot \n_x R_{k}(M_k^{-\frac{p_k-1}{2}}y)|\leq C M_k^{-(p_k-1)}|y|^2.
\end{align}
On the other hand, similar to \eqref{v_k^2}, for $|y|\leq \s M_k^{\frac{p_k-1}{2}}$ and large $k$, $v_k^2(y) \leq C(1+|y|^2)^{-2}$.
Then, for large $k$, we have the estimate
\begin{align}
\begin{split}
M_k^2|I_{k,4}| \leq&C M_k^{6-2p_k}\int_{\{|y|\leq \s M_k^{\frac{p_k-1}{2}}\}/\Gamma}(|y\cdot\n_y \tilde R_k|+2|\tilde R_k|)v_k^2\\
\leq& C  M_k^{6-2p_k} M_k^{-(p_k-1)} \int_{\{|y|\leq \s M_k^{\frac{p_k-1}{2}}\}/\Gamma} |y|^2(1+|y|^2)^{-2}\\
\leq & C M_k^{7-3p_k} \int_0^{\s M_k^{\frac{p_k-1}{2}}}r^5(1+r^2)^{-2}dr\\
\leq& C M_k^{7-3p_k} ((\s M_k^{\frac{p_k-1}{2}})^2+ C_1) = C\s^2 M_k^{2\tau_k} + CC_1M_k^{7-3p_k}\leq C\s^2\ \ \mathrm{as\ }k\to\infty.
\end{split}
\end{align}
\end{proof}
Next, we estimate the most important term $M_k^2I_{k,5}$.
\begin{lemma}\label{lemma_i5}
There exists a constant $0<\delta <1$ such that when $\s < \delta$,
\begin{align}
\label{I_5est}
\lim_{k\to\infty}M_k^2I_{k,5}=\frac{2\Delta_xK_\infty(\bar x)\cdot Vol(S^3)}{3|\Gamma|K_\infty(\bar x)^2}.
\end{align}
\end{lemma}
\begin{proof} First, note that $\lim_{k\to\infty}f_k^{-\tau_k}(x) = 1$, and
\begin{align}
\lim_{k\to\infty}|\nabla_xf_k^{-\tau_k}(x)| = \lim_{k\to\infty}-\tau_kf_k^{-\tau_k-1}(x)|\nabla_xf_k(x)| = 0
\end{align}
uniformly for $|x|\leq \s$. It follows
\begin{align}
\begin{split}
\lim_{k\to\infty}M_k^2I_{k,5}= & \frac{1}{p_k+1}M_k^2\int_{\{|x|\leq \s\}/\Gamma}(x\cdot \nabla_x(K_kf_k^{-\tau_k}))u_k^{p_k+1}\\
=&\frac{1}{p_k+1}M_k^2\int_{\{|x|\leq \s\}/\Gamma}(x\cdot \nabla_xK_k)u_k^{p_k+1}\\
=&\frac{1}{p_k+1}M_k^{2+\tau_k}\int_{\{|y|\leq \s M_k^{\frac{p_k-1}{2}}\}/\Gamma}(y\cdot \nabla_{y}\tilde K_k)v_k^{p_k+1}.
\end{split}
\end{align}
There exists a constant $0<\delta <1$ such that when $\s < \delta$, by power series expansion and Proposition \ref{grad0}, we have
\begin{align}
\begin{split}
K_k(x) = &K_k(0) + (K_k)_{,i}(0)x^i + \frac{1}{2}(K_k)_{,ij}(0)x^ix^j + O(|x|^3)\\
=&K_k(0) + \frac{1}{2}(K_k)_{,ij}(0)x^ix^j + O(M_k^{-2+2\delta}|x|)+O(|x|^3),
\end{split}
\end{align}
for $|x|\leq \s$, where $(K_k)_{,ij}(0)$ denotes the second order partial derivatives of $K_k$ in coordinates $\{x^i\}$ at the point $x_k$. It is not hard to verify that
\begin{align}
\begin{split}
&x\cdot \nabla_x[K_k(0) + O(M_k^{-2+2\delta}|x|)+O(|x|^3)] =O( M_k^{-2+2\delta}|x|)+O(|x|^3),\\
&x\cdot\nabla_x\Big(\frac{1}{2}(K_k)_{,ij}(0)x^ix^j\Big) = (K_k)_{,ij}(0)x^ix^j.
\end{split}
\end{align}
Thus
\begin{align}
x\cdot \nabla_xK_k(x) =  (K_k)_{,ij}(0)x^ix^j + O( M_k^{-2+2\delta}|x|)+O(|x|^3),
\end{align}
which implies
\begin{align}
y\cdot \nabla_y\tilde K_k(y) =  M_k^{-(p_k-1)}(K_k)_{,ij}(0)y^iy^j + O( M_k^{-2-\frac{p_k-1}{2}+2\delta}|y|)+O(M_k^{-3\frac{p_k-1}{2}}|y|^3).
\end{align}
Hence
\begin{align}
\label{weak_power_est}
|y\cdot \nabla_y\tilde K_k(y)| \leq  C M_k^{-(p_k-1)}(1+|y|)^2
\end{align}
and
\begin{align}
\label{strong_power_est}
|y\cdot \nabla_y\tilde K_k(y) -  M_k^{-(p_k-1)}(K_k)_{,ij}(0)y^iy^j|\leq C M_k^{-3\frac{p_k-1}{2}}(1+|y|)^3,
\end{align}
for $|y|\leq \s M_k^{\frac{p_k-1}{2}}$ and large $k$.

On the other hand, by power series expansion and Proposition \ref{vUdiff},
\begin{align}
\begin{split}
|v_k^{p_k+1}(y) - U_{K_k(x_k)}^{p_k+ 1}(y)|\leq& C\cdot (p_k+1)U_{K_k(x_k)}^{p_k}(y)|v_k(y) - U_{K_k(x_k)}(y)|\\
\leq& C (1+|y|^2)^{-p_k}M_k^{-2}.
\end{split}
\end{align}
Together with estimate \eqref{weak_power_est}, for large $k$, we have
\begin{align}
\begin{split}
&M_k^{2+\tau_k}\int_{\{|y|\leq \s M_k^{\frac{p_k-1}{2}}\}/\Gamma}(y\cdot \nabla_{y}\tilde K_k)|v_k^{p_k+1}- U_{K_k(x_k)}^{p_k+ 1}|\\
\leq& C M_k^{2+\tau_k}M_k^{-(p_k-1)}M_k^{-2}\int_{\{|y|\leq \s M_k^{\frac{p_k-1}{2}}\}/\Gamma}(1+|y|)^2(1+|y|^2)^{-p_k}\\
\leq& C M_k^{-(p_k-1)+\tau_k}\int_0^{\s M_k^{\frac{p_k-1}{2}}}r^3(1+r)^2(1+r^2)^{-3+o(1)}dr\\
\leq&C M_k^{-(p_k-1)+\tau_k} (\s M_k^{\frac{p_k-1}{2}} + C_1) \leq C M_k^{-\frac{p_k-1}{2}+\tau_k} \to 0\ \ \mathrm{as\ }k\to\infty.
\end{split}
\end{align}
Therefore,
\begin{align}
\lim_{k\to\infty}M_k^2I_{k,5}= \lim_{k\to\infty}\frac{1}{p_k+1}M_k^{2+\tau_k}\int_{\{|y|\leq \s M_k^{\frac{p_k-1}{2}}\}/\Gamma}(y\cdot \nabla_{y}\tilde K_k)U_{K_k(x_k)}^{p_k+1}.
\end{align}
Using estimate \eqref{strong_power_est}, for large $k$, we have
\begin{align}
\begin{split}
&M_k^{2+\tau_k}\int_{\{|y|\leq \s M_k^{\frac{p_k-1}{2}}\}/\Gamma}|y\cdot \nabla_y\tilde K_k(y) -  M_k^{-(p_k-1)}(K_k)_{,ij}(0)y^iy^j|\cdot U_{K_k(x_k)}^{p_k+1}\\
\leq&C M_k^{2+\tau_k}M_k^{-3\frac{p_k-1}{2}}\int_{\{|y|\leq \s M_k^{\frac{p_k-1}{2}}\}/\Gamma}(1+|y|)^3(1+|y|^2)^{-4+o(1)}\\
\leq & C M_k^{2 - 3\frac{p_k-1}{2} + \tau_k}\int_0^{\s M_k^{\frac{p_k-1}{2}}}r^3(1+r)^3(1+r^2)^{-4+o(1)}dr\\
\leq&  C M_k^{2 - 3\frac{p_k-1}{2} + \tau_k}((\s M_k^{\frac{p_k-1}{2}})^{-1+o(1)} + C_1)\to0\ \ \mathrm{as\ }k\to 0.
\end{split}
\end{align}
Therefore,
\begin{align}
\lim_{k\to\infty}M_k^2I_{k,5}= \lim_{k\to\infty}\frac{1}{p_k+1}M_k^{2\tau_k}\int_{\{|y|\leq \s M_k^{\frac{p_k-1}{2}}\}/\Gamma}(K_k)_{,ij}(0)y^iy^jU_{K_k(x_k)}^{p_k+1}.
\end{align}
Since $U_{K_k(x_k)}$ is a radial function,
\begin{align}
\int_{\{|y|\leq \s M_k^{\frac{p_k-1}{2}}\}/\Gamma}y^iy^jU_{K_k(x_k)} = \left\{\begin{array}{ll}0,&i\neq j,\\ \frac{1}{4}\int_{\{|y|\leq \s M_k^{\frac{p_k-1}{2}}\}/\Gamma}|y|^2U_{K_k(x_k)},&i=j.\end{array}\right.
\end{align}
Thus we obtain
\begin{align}
\lim_{k\to\infty}M_k^2I_{k,5}=& \lim_{k\to\infty}\frac{1}{p_k+1}M_k^{2\tau_k}\int_{\{|y|\leq \s M_k^{\frac{p_k-1}{2}}\}/\Gamma}\frac{1}{4}\Delta_xK_k(0)\cdot |y|^2U_{K_k(x_k)}^{p_k+1},
\end{align}
where $\Delta_xK_k(0) = \Delta_xK_k(x_k)$ is the Laplacian of $K_k$ in coordinates $\{x^i\}$ at the point $x_k$. Because $\Delta_xK_k(x_k)\to \Delta_xK_\infty(\bar x)$ and $U_{K_k(x_k)}^{p_k+1}\to U_{K_\infty(\bar x)}^{4}$ uniformly on $\RR^4/\Gamma$, we have
\begin{align}
\begin{split}
\lim_{k\to\infty}M_k^2I_{k,5}=& \frac{1}{16}\int_{y\in\RR^4/\Gamma}\Delta_xK_\infty(\bar x)\cdot |y|^2\Big(\frac{8}{K_\infty(\bar x)}\Big)^2(1+|y|^2)^{-4}\\
=&\frac{4\Delta_xK_\infty(\bar x)\cdot Vol(S^3)}{|\Gamma|K_\infty(\bar x)^2}\int_0^\infty r^5(1+r^2)^{-4}dr\\
=&\frac{2\Delta_xK_\infty(\bar x)\cdot Vol(S^3)}{3|\Gamma|K_\infty(\bar x)^2}.
\end{split}
\end{align}
\end{proof}
Finally,  we have achieved the following proposition.
\begin{proposition}\label{pohoprop}
Assuming Condition \ref{blow-up_cond}, we have the following inequality
\begin{align}
\lim_{\s\to0}P(\s, h_{\bar x}) \geq \frac{2\Delta_gK_\infty(\bar x)\cdot Vol(S^3)}{3|\Gamma|K_\infty(\bar x)^2}.
\end{align}
Moreover, if $\tau_k = 0$ for all $k$, we have the equality
\begin{align}
\lim_{\s\to0}P(\s, h_{\bar x}) = \frac{2\Delta_gK_\infty(\bar x)\cdot Vol(S^3)}{3|\Gamma|K_\infty(\bar x)^2}.
\end{align}
Here, $\Delta_g$ is the Laplacian with respect to the limit conformal normal metric $g$.
\end{proposition}
\begin{proof}
By Lemma \ref{lemma_i1} -- Lemma \ref{lemma_i5} and equation \eqref{pohozaev},
\begin{align}
\lim_{\s\to0}\limsup_{k\to\infty}M_k^2I_{k,7} = \frac{2\Delta_gK_\infty(\bar x)\cdot Vol(S^3)}{3|\Gamma|K_\infty(\bar x)^2} + \lim_{\s\to0}\limsup_{k\to\infty}M_k^2I_{k,6}.
\end{align}
By Corollary \ref{greenlimit},
\begin{align}
\begin{split}
&\limsup_{k\to\infty}M_k^2I_{k,7}\\
=&\limsup_{k\to\infty} \int_{\{|x|=\s\}/\Gamma} \Big\{\Big(\Big|\frac{\p (M_ku_k)}{\p \nu_x}\Big|^2 -\frac{1}{2}|\n_x (M_ku_k)|^2\Big)\s+(M_ku_k)\frac{\p (M_ku_k)}{\p\nu_x}\Big\}\\
=&P(\s, h_{\bar x}).
\end{split}
\end{align}
On the other hand, recall that
\begin{align}
I_{k,6} = &\frac{\tau_k}{p_k+1}\int_{\{|x|\leq \s\}/\Gamma}K_kf_k^{-\tau_k}u_k^{p_k+1}.
\end{align}
It's clear that $I_{k,6}\geq 0$, and $I_{k,6} = 0$ if $\tau_k=0$. Our proposition is proved.
\end{proof}

\subsection{Green Function}
Assuming Condition \ref{blow-up_cond}, Corollary \ref{greenlimit} tells us that $M_ku_k \to h_{\bar x} = aG(\cdot, \bar x)$ in $C^2_{loc}((B_r(\bar x)-\{\bar x\})/\Gamma)$. We are going to determine the constant $a$ and the regular part of $G(\cdot, \bar x)$ evaluated at $\bar x$, where $G(\cdot, \bar x)$ is the standard Green function for $L_g = \Delta_g - \frac{1}{6}R_g$ at the point $\bar x$, and $g$ is the limit metric, i.e. for any $C^2$ function $\phi$ and small $\s > 0$,
\begin{align}
\int_{\{|x|\leq \s\}/\Gamma}G(L_g\phi)dV_g=\phi(\bar x).
\end{align}
Hence Corollary \ref{greenlimit} implies that
\begin{align}
\lim_{k\to\infty}\int_{\{|x|\leq \s\}/\Gamma}(M_ku_k)(L_{g_k}\phi)dV_{g_k} = a \phi(\bar x).
\end{align}
\begin{lemma}\label{lemma_aterm}
The constant $a$ in Corollary \ref{greenlimit} is
\begin{align}
a = -\frac{4\sqrt{2}Vol(S^3)}{|\Gamma|\sqrt{K_\infty(\bar x)}}.
\end{align}
\end{lemma}
\begin{proof}
Recall Remark~\ref{coord_setting} and the change of variables in \eqref{cov}, \eqref{res_func}. For any $C^2$ function $\phi$, define $\tilde \phi(y) = \phi(M_k^{-\frac{p_k-1}{2}}y)$. Integrating by parts, we obtain
\begin{align}
\begin{split}
&\int_{\{|x|\leq \s\}/\Gamma}(M_ku_k)(L_{g_k}\phi)\\
 = &\int_{\{|x|\leq \s\}/\Gamma}\Big(\Delta_x - \frac{1}{6}R_k\Big)(M_ku_k)\cdot \phi\\
=&M_k^{-2(p_k-1)}\int_{\{|y|\leq \s M_k^{\frac{p_k-1}{2}}\}/\Gamma}\Big(M_k^{p_k-1}\Delta_y-\frac{1}{6}\tilde R\Big)(M_k^2v_k)\cdot\tilde\phi\\
=&\Big(M_k^{\tau_k}\int_{\{|y|\leq \s M_k^{\frac{p_k-1}{2}}\}/\Gamma}(\Delta_yv_k)\cdot\tilde\phi\Big) -\Big(\frac{1}{6}M_k^{4-2p_k}\int_{\{|y|\leq M_k^{\frac{p_k-1}{2}}\s\}/\Gamma}\tilde R_k v_k\tilde \phi\Big).
\end{split}
\end{align}
Since $\lim_{k\to\infty}M_k^{4-2p_k}=0$, the second integral vanishes when $k\to\infty$. Thus
\begin{align}
\begin{split}
a\phi(\bar x) =&\lim_{\s\to0}\lim_{k\to\infty}\int_{\{|x|\leq \s\}/\Gamma}(M_ku_k)(L_{g_k}\phi)\\
 = &\lim_{\s\to 0}\lim_{k\to\infty}\int_{\{|y|\leq M_k^{\frac{2}{n-2}}\s\}/\Gamma}(\Delta_yv_k)\cdot\phi(M_k^{-\frac{p_k-1}{2}}y)\\
=&\Big(\int_{\RR^n/\Gamma}\Delta_yU_{K_\infty(\bar x)}(y)dy\Big)\cdot \phi(\bar x),
\end{split}
\end{align}
where the last equality is by Proposition \ref{p:m43}. Therefore the constant $a$ is
\begin{align}
\begin{split}
a = \int_{\RR^n/\Gamma}\Delta_yU_{K_\infty(\bar x)}(y)dy
&=-\int_{\RR^n/\Gamma}K_\infty(\bar x)U_{K_\infty(\bar x)}^3(y)dy\\
=&-\frac{16\sqrt{2}}{\sqrt{K_\infty(\bar x)}}\int_{\RR^n/\Gamma}(1+|y|^2)^{-3}dy\\
=&-\frac{16\sqrt{2}Vol(S^3)}{|\Gamma|\sqrt{K_\infty(\bar x)}}\int_{0}^\infty r^3(1+r^2)^{-3}dr =-\frac{4\sqrt{2}Vol(S^3)}{|\Gamma|\sqrt{K_\infty(\bar x)}}.
\end{split}
\end{align}
\end{proof}

On the other hand, recall that we assume for each $k$, $g_k$ is the conformal normal metric with conformal normal coordinates centered at point $x_k$, hence the limit metric $g$ is the conformal normal metric with conformal normal coordinates centered at point $\bar x$. Thus we have the following proposition.
\begin{proposition}
\label{prop_Abterm}
In $g$-conformal normal coordinates $\{x^i\}$ centered at $\bar x$, $G(\cdot, \bar x)$ has the expansion
\begin{align}
G(\cdot, \bar x) = b\psi_{\bar x}= b[r^{-2} + A_{\bar x} + O(r)],
\end{align}
where $r = |x|$, $\psi_{\bar x}$ is from Definition \ref{def_psi},
$A_{\bar x}$ is a constant, and
\begin{align}
\label{b}
b = -\frac{|\Gamma|}{2Vol(S^{3})}.
\end{align}
\end{proposition}
\begin{proof}
The proof is given by \cite[Definition 6.2 and Lemma 6.4]{Lee-Parker}. Note that our notation is different from \cite{Lee-Parker}. Our point $\bar x$ is their point $P$; our operator $L_g$ is equal to $-1/6$ multiplied with their box operator $\Box$; our $G(\cdot, \bar x)$ is equal to $-6$ multiplied with their $\Gamma_P$, our $\psi_{\bar x}$ is their $G$. And the $\Gamma$ in our equation \eqref{b} is the quotient group near point $\bar x$.\\
\end{proof}
Then we can relate the Pohozaev identity with the constant term $A_{\bar x}$.
\begin{lemma}\label{lemma_pohoA}
We have
\begin{align}
\lim_{\s\to0}P(\s, h_{\bar x}) = -\frac{16Vol(S^3)}{|\Gamma|K_\infty(\bar x)}\cdot A_{\bar x}.
\end{align}
\end{lemma}
\begin{proof} We will write $G$ instead of $G(\cdot, \bar x)$. Using $h_{\bar x} = a G$, we have
\begin{align}
P(\s, h_{\bar x}) = a^2\int_{\{|x|=\s\}/\Gamma} \Big\{\Big(\Big|\frac{\p G}{\p \nu}\Big|^2 -\frac{1}{2}|\n G|^2\Big)\s+G\frac{\p G}{\p\nu_x}\Big\}.
\end{align}
By Proposition \ref{prop_Abterm}, on $|x| = \s$ for small $\s$,
\begin{align}
&G= b[\s^{-2} + A_{\bar x} + O(\s)], \quad \Big|\frac{\p G}{\p \nu}\Big|=b[-2\s^{-3} + O(1)],\quad |\nabla G|=b[-2\s^{-3} + O(1)].
\end{align}
Then we have
\begin{align}
\notag
\Big(\Big|\frac{\p G}{\p \nu}\Big|^2 -\frac{1}{2}|\n G|^2\Big)\s +G\frac{\p G}{\p\nu_x}
&=b^2\Big(\frac{\s}{2}\cdot[-2\s^{-3}+O(1)]^2+[\s^{-2}+A_{\bar x}+O(\s)]\cdot[-2\s^{-3}+O(1)]\Big)\\
&=b^2\Big(\frac{\s}{2}\cdot[4\s^{-6}+O(\s^{-3})]-2\s^{-5}-2A_{\bar x}\s^{-3}+O(\s^{-2})\Big)\\
\notag
&=-2b^2A_{\bar x}\s^{-3} + O(\s^{-2}).
\end{align}
Hence
\begin{align}
P(\s, h_{\bar x}) =& a^2\int_{\{|x|=\s\}/\Gamma} \Big(-2b^2A_{\bar x}\s^{-3} + O(\s^{-2})\Big)
=-\frac{2a^2b^2A_{\bar x}\cdot Vol(S^3)}{|\Gamma|} + O(\s).
\end{align}
Letting $\s \to0$, using Lemma \ref{lemma_aterm} and Proposition \ref{prop_Abterm}, we have
\begin{align}
\begin{split}
\lim_{\s\to0}P(\s, h_{\bar x}) =& -\frac{2a^2b^2A_{\bar x}\cdot Vol(S^3)}{|\Gamma|}\\
=&(-2)\cdot\frac{Vol(S^3)}{|\Gamma|}\cdot \frac{32Vol(S^3)^2}{|\Gamma|^2K_\infty(\bar x)}\cdot \frac{|\Gamma|^2}{4Vol(S^3)^2}\cdot A_{\bar x} =-\frac{16Vol(S^3)}{|\Gamma|K_\infty(\bar x)}\cdot A_{\bar x}.
\end{split}
\end{align}
\end{proof}
\begin{proposition}\label{Aprop}
Assuming Condition \ref{blow-up_cond}, we have the following inequality
\begin{align}
A_{\bar x}\leq -\frac{\Delta_gK_\infty(\bar x)}{24K_\infty(\bar x)}.
\end{align}
Moreover, if $\tau_k = 0$ for all $k$, we have the equality
\begin{align}
A_{\bar x} = -\frac{\Delta_gK_\infty(\bar x)}{24K_\infty(\bar x)}.
\end{align}
Here, $\Delta_g$ is the Laplacian with respect to the limit conformal normal metric $g$. 
\end{proposition}
\begin{proof}
Assuming Condition \ref{blow-up_cond}, by Proposition \ref{pohoprop}, we have
\begin{align}
\lim_{\s\to0}P(\s, h_{\bar x}) \geq \frac{2\Delta_gK_\infty(\bar x)\cdot Vol(S^3)}{3|\Gamma|K_\infty(\bar x)^2}.
\end{align}
By Lemma \ref{Aprop}, we have
\begin{align}
\lim_{\s\to0}P(\s, h_{\bar x}) = -\frac{16Vol(S^3)}{|\Gamma|K_\infty(\bar x)}\cdot A_{\bar x}.
\end{align}
Hence
\begin{align}
A_{\bar x}\leq -\frac{2\Delta_gK_\infty(\bar x)\cdot Vol(S^3)}{3|\Gamma|K_\infty(\bar x)^2}\cdot \frac{|\Gamma|K_\infty(\bar x)}{16Vol(S^3)} = -\frac{\Delta_gK_\infty(\bar x)}{24K_\infty(\bar x)}.
\end{align}
The case that $\tau_k = 0$ for all $k$ follows similarly.
\end{proof}
Moreover, by relating $A_{\bar x}$ with the mass, we can remove the assumption "conformal normal metric", as follows.
\begin{proposition}
\label{mass_prop}
Assuming Condition \ref{blow-up_cond}, let $g = \lim_{k\to \infty}g_k$ be the limit metric, but do not necessarily assume that $g_k, g$ are conformal normal metrics. Let
$\hat g_{\bar x} = \psi^2_{\bar x} g$
be the conformal blow-up of $g$ at the point $\bar x$, as in Definition \ref{def_psi}.
We have the following inequality
\begin{align}
m(\hat g_{\bar x}) \leq -\frac{\Delta_gK_\infty(\bar x)}{2K_\infty(\bar x)}.
\end{align}
Moreover, if $\tau_k = 0$ for all $k$, we have the equality
\begin{align}
m(\hat g_{\bar x}) = -\frac{\Delta_gK_\infty(\bar x)}{2K_\infty(\bar x)}.
\end{align}
\end{proposition}
\begin{proof}
For each $k$, let $r_k$ denote the $g_k$-distance function from the point $x_k$. Assume
$\tilde g_k = \phi_k^2 g_k$
is the conformal normal metric with conformal normal coordinates $\{\tilde x^i\}$ centered at point $x_k$. By \cite[Theorem 5.6]{Lee-Parker}, after applying a dilation and a translation to the coordinates $\{\tilde x^i\}$, we may assume $\phi_k(x_k) = 1$ and $\nabla\phi_k(x_k)=0$, in other words, for small $r_k$,
$\phi_k = 1 + O(r_k^2)$. When $k\to 0$, we have that
$\tilde g = \phi^2 g$,
where $g, \tilde g, \phi$ are the limit of $g_k, \tilde g_k, \phi_k$ as $k\to\infty$. Moreover, $\phi = 1 + O(r^2)$, where $r$ is the $g$-distance function from the point $\bar x$.
Next, recall that the conformal transformation law of the Laplacian for $\tilde g = e^{2\varphi} g$ is
\begin{align}
\label{trans_lap}
\Delta_{\tilde g} = e^{-2\varphi}\Delta_g + (n-2)e^{-2\varphi}g^{ij}\frac{\p \varphi}{\p x_j}\frac{\p}{\p x_i}.
\end{align}
Thus, for any $f\in C^2$, we have
$\Delta_{\tilde g}f(\bar x) = \Delta_{g}f(\bar x)$.
Let $A_{\bar x}$ be the regular part corresponding to conformal blow-up of $\tilde g$ at point $\bar x$. Let $\hat g_{\bar x}$ be the conformal blow-up of $g$ at point $\bar x$. By \cite[Lemma~9.7]{Lee-Parker}, we have 
\begin{align}
\label{m=12A}
m(\hat g_{\bar x}) = 12 A_{\bar x}.
\end{align}
Therefore, we know
\begin{align}
\begin{split}
&m(\hat g_{\bar x}) \leq -\frac{\Delta_gK_\infty(\bar x)}{2K_\infty(\bar x)}\ \ \Longleftrightarrow\ \ A_{\bar x}\leq -\frac{\Delta_{\tilde g}K_\infty(\bar x)}{24K_\infty(\bar x)},\\
&m(\hat g_{\bar x}) = -\frac{\Delta_gK_\infty(\bar x)}{2K_\infty(\bar x)}\ \ \Longleftrightarrow\ \ A_{\bar x}= -\frac{\Delta_{\tilde g}K_\infty(\bar x)}{24K_\infty(\bar x)},
\end{split}
\end{align}
which implies that this proposition is equivalent to Proposition \ref{Aprop}.
\end{proof}

\section{Blow-up points must be isolated and simple}
\label{s:blow-up}
In \cite{Li-Zhu}, they proved that on a 3-dimensional compact manifold, all blow-up points for $\eqref{confseq}$ must be isolated simple blow-up points. The same result in higher dimensions is proved by \cite{KMS}, \cite{Li-Zhang}, but only for constant prescribed scalar curvature. Here, we will modify their proofs to show that the same result holds on 4-dimensional compact orbifolds, for a sequence of variable prescribed scalar curvatures.

Let $(M, g)$ be a compact Riemannian 4-dimensional orbifold with singularities
\begin{align}
\Sigma_\Gamma = \{(q_1,\Gamma_1),\cdots,(q_l,\Gamma_l)\}
\end{align}
and positive scalar curvature $R_g$. Assume $u$ is a positive $C^2$ solution of equation \eqref{conformalequ} on $M$, where $K$ is a positive $C^2$ function and $f$ is a positive $C^1$ function. For any point $\bar x\in M$, define $\Omega_{\bar x,\s}$ in the following:

a) if $\bar x$ is a smooth point, define $\Omega_{\bar x,\s} = B_\s(\bar x)$ for some $\s>0$ such that its closure $\bar \Omega_{\bar x,\s}$ doesn't include any singular point. In other words, $d_g(\bar x, \{q_1,\cdots,q_l\}) > \s$;

b) if $\bar x = q_j$ for some $1\leq j\leq l$, choose $\s = \s_j$ where $\s_j$ is as defined in Definition \ref{obf_def}. Then the neighborhood of $\bar x$ is a quotient ball $B_\s(\bar x)/\Gamma$. Define $\Omega_{\bar x,\s} = \pi_j^*(B_\s(\bar x)/\Gamma)$ to be the lifting-up space. Denote the lifting-up functions and metric still by $u$, $f$, $K$ and $g$.

First, let us recall a lemma from \cite{Li-Zhu}.
\begin{lemma}[\cite{Li-Zhu} Lemma 5.1]
\label{LZ_lemma5.1}
Let $(M,g), u, K, f, \Omega_{\bar x, \s}$ be as defined above. Given any small $\varepsilon >0$ and large $R'>1$, there exists a large positive $C_0$, depending only on $M, g, \Vert f\Vert_{C^1(M)}$, $\inf_M K, \Vert K\Vert_{C^2(M)}, \varepsilon$ and $R'$ such that for any compact $S\subset \bar \Omega_{\bar x, \s}$, if $u$ satisfies
\begin{align}
\label{max_u_cond}
\max_{x\in \bar \Omega_{\bar x, \s}\setminus S} d_g(x, S)^{\frac{2}{p-1}}u(x) \geq C_0,
\end{align}
then we have $p > 3-\varepsilon$ and for some local maximum point of $u$ in $\Omega_{\bar x, \s}\setminus S$, denoted as $x_0$,
\begin{align}
\Big\Vert u(x_0)^{-1}u\Big(\exp_{x_0}\Big(u(x_0)^{-\frac{p-1}{2}}x\Big)\Big) - U_{K(x_0)}(x)\Big\Vert_{C^2(|x|<2R')}<\varepsilon,
\end{align}
where $d_g(x,S)$ denotes the distance of $y$ to $S$, and $d_g(x,S)=1$ if $S=\emptyset$.
\end{lemma}
\begin{proof}
The case that $K$ is a positive constant and $\Omega_{\bar x, \s}$ is replaced by a compact 3-dimensional manifold $M$ is proved in \cite[Lemma 5.1]{Li-Zhu}. As mentioned in \cite[Section~7]{Li-Zhu}, it can be generalized to the case $K$ is a variable function. With a straightforward modification, it can also be generalized to higher dimensions.
Moreover, note that the statement and proof in \cite[Lemma 5.1]{Li-Zhu} is just a local argument in $M\setminus S$. Thus their proof remains valid in our case by choosing their $S$ to be any compact subset of $M$ satisfying $M\setminus S\subset B_\s(\bar x)/\Gamma$, and considering the lifting-up of $M\setminus S$ if $\bar x$ is a singular point.
\end{proof}

Using Lemma \ref{LZ_lemma5.1}, we can prove the following.
\begin{proposition}[\cite{Li-Zhu} Proposition 5.1]
\label{LZ_prop5.1}
Let $(M,g), u, K, f, \Omega_{\bar x, \s}$ be as defined in the beginning of this section. Given small $\varepsilon > 0$ and large $R'$, there exist some positive constants $C_0$ and $C_1$ depending on $M, g, \Vert f\Vert_{C^1(M)}, \inf_M K, \Vert K\Vert_{C^2(M)}, \varepsilon$ and $R'$ such that if
\begin{align}
\max_{\Omega_{\bar x, \s}} u > C_0,
\end{align}
then there exists some integer $N = N(u) \geq 1$ and $N$ local maximum points of $u$ denoted as $\{x_1,\cdots, x_N\}\subset \Omega_{\bar x,\s}$, such that:\\
1) $3-\varepsilon <p\leq 3$,\\
2) $\overline{B_{r_i}(x_i)}\cap \overline{B_{r_j}(x_j)}=\emptyset$ for $i\neq j$, where $r_j = R' u(x_j)^{-\frac{p-1}{2}}$,\\
and for each $j$,
\begin{align}
\Big\Vert u(x_j)^{-1}u\Big(\exp_{x_j}\Big(u(x_j)^{-\frac{p-1}{2}}x\Big)\Big) - U_{K(x_j)}(x)\Big\Vert_{C^2(|x|<2R')}<\varepsilon,
\end{align}
3) $d_g(x_i,x_j)^{\frac{2}{p-1}}u(x_j)\geq C_0$ for $j>i$, while $d_g(x, \{x_1,\cdots, x_N\})^{\frac{2}{p-1}}u(x)\leq C_1$ for all $x\in \Omega_{\bar x, \s}$.
\end{proposition}
\begin{proof}
The case that $K$ is a positive constant and $\Omega_{\bar x, \s}$ is replaced by a compact 3-dimensional manifold $M$ is proved in \cite[Proposition 5.1]{Li-Zhu}.
Briefly, that proof was completed by an induction process as following: first apply \cite[Lemma 5.1]{Li-Zhu} with $S=\emptyset$ to get the first local maximum point of $u$, denoted by $x_1$; assuming we have already got local maximum points $\{x_1,\cdots, x_l\}$, apply \cite[Lemma 5.1]{Li-Zhu} with $S=\cup_{j=1}^l\overline{B_{r_j}(x_j)}$ to get the next local maximum point. If the condition $\eqref{max_u_cond}$ is not satisfied in any step, the process stops. This process must stop after a finite number of times because each time $\int_{B_{r_j}(x_j)}|\nabla u|^2$ is greater than some positive universal constant and their sum $\sum_j\int_{B_{r_j}(x_j)}|\nabla u|^2\leq \int_M|\nabla u|^2$ which is finite. It is not hard to verify the set $\{x_1,\cdots,x_N\}$ constructed by the process satisfies all the above properties. In our case, we can continue the same process by inductively applying our Lemma \ref{LZ_lemma5.1} with $S$ as mentioned above. Everything follows the same way and our proposition is proved.
\end{proof}
Next, we can rule out bubble accumulation.
\begin{proposition}\label{iso}
Let $(M,g), u, K, f, \Omega_{\bar x, \s}$ be as defined in the beginning of this section. Let $\varepsilon, R', C_0, C_1$ and $\{x_1,\cdots, x_N\}$ be as defined in Proposition \ref{LZ_prop5.1}. If $\varepsilon$ is sufficiently small and $R'$ is sufficiently large, then there exists a positive constant $\bar C$, which only depends on $M$, $g$, $\Vert f\Vert_{C^1(M)}$, $\inf_M K$, $\Vert K\Vert_{C^2(M)}$, $\varepsilon$ and $R'$, such that if
$\max_{\Omega_{\bar x, \s}} u \geq C_0$,
then
$d_g(x_j, x_l) \geq \bar C$,
for all $j\neq l$.
\end{proposition}
\begin{proof}
The proof is similar to \cite[Proposition 8.3]{KMS} and \cite[Proposition 5.2]{Li-Zhu}.
We will prove it by contradiction. Suppose that such a constant $\bar C$ does not exist, then there exist sequences $p_k\to p\in(3-\varepsilon,3]$ and $\{u_k\}$ with $\max_{\Omega_{\bar x, \s}} u_k\geq C_0$ and
\begin{align}
\lim_{k\to\infty} \min_{j\neq l}d_g(x_j(u_k), x_l(u_k)) = 0.
\end{align}
Without loss of generality, assume that
\begin{align}
\delta_k = d_g(x_1(u_k), x_2(u_k)) = \min_{j\neq l}d_g(x_j(u_k), x_l(u_k))\to 0\ \mathrm{as\ }k\to\infty.
\end{align}
For each $k$, take normal coordinates $\{x^i\}$ centered at point $x_1(u_k)$ and consider change of variables $y = \delta_k^{-1}x$. Rescale $u_k$ by
\begin{align}
v_k(y) = \delta_k^{\frac{2}{p_k-1}}u_k(\delta_k y),\ \ \forall |y| < \delta_k^{-1}.
\end{align}
Then $v_k$ satisfies
\begin{align}
L_{h_k} v_k + \tilde K_k\tilde f_k^{-\tau_k}v_k^{p_k} = 0,
\end{align}
where $(h_k)_{ij}(y) = g_{ij}(\delta_ky)$, $\tilde K_k(y) = K(\delta_ky)$ and $\tilde f_k(y) = f(\delta_ky)$. If $x_j(u_k)\in B_{\sqrt{\delta_k}}(x_1)$, denote by $y_j(u_k) = \delta_k^{-1}x_j(u_k)$ the $y$-coordinate of point $x_j(u_k)$. By the proof of \cite[Proposition 8.3]{KMS}, we have $y_1(u_k) = 0$, $y_2(u_k)\to \bar y_2$ with $|\bar y_2| = 1$, $\{0, \bar y_2\}$ are isolated simple blow-up points for $\{v_k\}$, and
\begin{align}
v_k(0)v_k(y)\to a_1(|y|^{-2} + b_1 + O(y))\ \ \mathrm{in\ }C^2_{loc}(\RR^4-S'),
\end{align}
where $S'$ denotes the set of blow-up points for $\{v_k\}$ and $a_1,b_1>0$ are some positive constants. On the other hand,
\begin{align}
\begin{split}
\tilde K_k(y) &= K(\delta_ky) \to K(0)\mathrm{\ in\ }C^0_{loc}\mathrm{\ norm},\\
|\nabla_y\tilde K_k(y)| &= \delta_k|\nabla_xK(\delta_ky)| \to 0\mathrm{\ in\ }C^0\mathrm{\ norm},\\
|\nabla^2_y\tilde K_k(y)| &= \delta_k^2|\nabla^2_xK(\delta_ky)| \to 0\mathrm{\ in\ }C^0\mathrm{\ norm}.
\end{split}
\end{align}
Hence $\tilde K_k$ converges to the constant $K(0)$ in the $C^2_{loc}$ norm, where $K(0)$ by definition is the $K$ value at the limit point of $x_1(u_k)$ as $k\to\infty$, possibly by passing to a subsequence. Applying Proposition~\ref{Aprop} to the blow-up sequence $\{v_k\}$, we get $b_1 \leq 0$, which contradicts $b_1 > 0$. Therefore our proposition is proved.
\end{proof}

\begin{corollary}
\label{blow_is_iso}
Let $(M,g)$ be a compact Riemannian 4-dimensional orbifold with positive scalar curvature. Suppose $\{f_k\}$ is a sequence of positive $C^1$ functions converging in the $C^1_{loc}$ topology to a positive function $f$. Also suppose $\{K_k\}$ is a sequence of positive $C^2$ functions converging in the $C^2_{loc}$ topology to a positive function $K_\infty$.
Let $\{u_k\}$ be a sequence of positive solutions of equation \eqref{confseq} on $M$ with $g_k = g$ and $\max_M u_k\to\infty$. Then $p_k\to 3$ and the set of blow-up points is finite and consists only of isolated blow-up points.
Moreover, if blow-up occurs at a singular point, i.e. $x_k \to \bar x$ and $u_k(x_k)\to \infty$ where $\bar x$ is a singular point, then there exists an integer $N\in\mathbb N^*$ such that for any $k> N$, $x_k = \bar x$.
\end{corollary}
\begin{proof}
By the assumption of $f_k$ and $K_k$, there exists a constant $C_2$ such that for large $k$,
\begin{align}
\Vert f_k\Vert_{C^1(M)} \leq C_2 \Vert f\Vert_{C^1(M)}, \quad
\inf_{M}K_k \geq C_2 \inf_M K_\infty, \quad
\Vert K_k\Vert_{C^2(M)} \leq C_2 \Vert K_\infty\Vert_{C^2(M)}.
\end{align}
By Proposition \ref{iso}, in each $\Omega_{x,\s}$, blow-up points must be isolated blow-up points and the number $N(u_k)$ as defined in Proposition \ref{LZ_prop5.1} must have an uniformly upper bound, otherwise, there cannot exist a constant $\bar C$ such that $d_g(x_i(u_k),x_j(u_k))\geq \bar C$ for all $i\neq j$ and $k$. A compact orbifold can be covered by finitely many sufficiently small open balls, which gives us finitely many $\Omega_{x,\s}$ as defined in the beginning of this section. Therefore the set of blow-up points on $M$ is finite and consists only of isolated blow-up points.

Moreover, if blow-up occurs at a singular point, i.e. $x_k \to \bar x$ and $u_k(x_k)\to \infty$ where $\bar x$ is a singular point associated with a nontrivial quotient group $\Gamma$. Suppose there exists a subsequence, still denoted by $(u_k, x_k)$, such that $x_k\neq \bar x$ for any $k$. Consider the lifting-up space $\Omega_{\bar x,\s}$ as defined in the beginning of this section. Let $\tilde x_k^{(1)},\cdots, \tilde x_k^{(|\Gamma|)}$ denote the lifting-up points of $x_k$. In the lifting-up space, it's clear that
\begin{align}
d_{\tilde g}(\tilde x_k^{(1)}, \tilde x_k^{(2)})\to 0\ \ \mathrm{as\ }k\to\infty,
\end{align}
which contradicts against Proposition \ref{iso}. Thus our corollary is proved.
\end{proof}

Given Corollary \ref{blow_is_iso}, we are able to conclude the following.
\begin{proposition}\label{iso_sim}
Let $(M,g), f_k, K_k, g_k$ be as defined in Corollary \ref{blow_is_iso}. Assume $u_k$ is a sequence of positive functions satisfying equation \eqref{confseq} and $x_k\to \bar x$ is an isolated blow-up point. Then $\bar x$ is an isolated simple blow-up point for $\{u_k\}$.
\end{proposition}
\begin{proof}
The proof is similar to \cite[Lemma 8.2]{KMS} and \cite[Proposition 4.1]{Li-Zhu}.
First, by Corollary~\ref{blow_is_iso}, without loss of generality, we may assume $x_k = \bar x$ for all $k$ if $\bar x$ is a singular point. Suppose $\bar x$ is an isolated blow-up point, but not an isolated simple blow-up point. Let $x = \{x^i\}$ be normal coordinates centered at $x_k$ and define the rescaled function
\begin{align}
v_k(y) = \tau_k^{\frac{2}{p_k-1}}u_k(\tau_k y),\ \ \forall |y| < \tau_k^{-1}.
\end{align}
Then $v_k$ satisfies
\begin{align}
L_{h_k} v_k + \tilde K_k\tilde f_k^{-\tau_k}v_k^{p_k} = 0,
\end{align}
where $(h_k)_{ij}(y) = (g_k)_{ij}(\tau_ky)$, $\tilde K_k(y) = K_k(\tau_ky)$ and $\tilde f_k(y) = f_k(\tau_ky)$. By the proof of \cite[Lemma 8.2]{KMS}, the origin $y=0$ is an isolated simple blow-up point for $\{v_k\}$ and
\begin{align}
v_k(0)v_k(y)\to h(y) = a_2(|y|^{-2} + b_2)\ \ \mathrm{in\ }C^2_{loc}((\RR^4-\{0\})/\Gamma),
\end{align}
where $a_2 =b_2 =1$,  where $\Gamma = \{e\}$ if $\bar{x}$ is a smooth point, but $\Gamma$ is the quotient group if $\bar x$ is a singular point.

On the other hand, because $K_k$ converges to $K_\infty$ in the $C^2_{loc}$ norm, we know
\begin{align}
\begin{split}
\tilde K_k(y) &= K_k(\tau_ky) \to K_\infty (\bar x)\mathrm{\ in\ }C^0_{loc}\mathrm{\ norm},\\
|\nabla_y\tilde K_k(y)| &= \tau_k|\nabla_xK_k(\tau_ky)| \to 0\mathrm{\ in\ }C^0_{loc}\mathrm{\ norm},\\
|\nabla^2_y\tilde K_k(y)| &= \tau_k^2|\nabla^2_xK_k(\tau_ky)| \to 0\mathrm{\ in\ }C^0_{loc}\mathrm{\ norm}.
\end{split}
\end{align}
Hence $\tilde K_k$ converges to the constant $K_\infty(\bar x)$ in the $C^2_{loc}$ norm.
Applying Proposition~\ref{Aprop} to the blow-up sequence $\{v_k\}$, we obtain
$b_2 \leq 0$, which contradicts $b_2 = 1$. Therefore $\bar x$ is an isolated simple blow-up points for $\{u_k\}$.
\end{proof}

\begin{corollary}
\label{blow_is_isosim}
Let $(M,g), f_k, K_k, g_k$ be as defined in Corollary \ref{blow_is_iso}. Assume $u_k$ is a sequence of positive functions satisfying equation \eqref{confseq}, then Condition~\ref{blow-up_cond} is a necessary condition for any blow-up point.
\end{corollary}
\begin{proof} This follows immediately from combining Corollary~\ref{blow_is_iso} and Proposition~\ref{iso_sim}. 
\end{proof}

\begin{proof}[Proof of Theorem \ref{cpt_crit_thm_combined}] To prove the upper bound $u\leq C$ under assumption \eqref{mass_ineq_sub} in Theorem~\ref{cpt_crit_thm_combined}, suppose the contrary. Then there exist $p_k\to 3$ and $\{u_k\}$ satisfying
\begin{align}
L_g u_k + K_ku_k^{p_k} = 0,
\end{align}
with $\max_M u_k\to\infty$, where $\{K_k\}$ is a sequence of positive $C^2$ functions converging in the $C^2_{loc}$ topology to a positive function $K_\infty$. Let $x_k$ denote the point where $u_k$ obtains a maximum, after possibly passing to a subsequence, we may assume $x_k\to \bar x$ is a blow-up point. By Corollary \ref{blow_is_isosim}, the sequence $(u_k,x_k)$ satisfies Condition \ref{blow-up_cond}, where $g_k$ in equation \eqref{confseq} is the metric conformal to $g$ with conformal normal coordinates centered at $x_k$. By Proposition \ref{mass_prop}, we know
\begin{align}
m(\hat g_{\bar x}) \leq -\frac{\Delta_gK_\infty(\bar x)}{2K_\infty(\bar x)},
\end{align}
which contradicts against assumption \eqref{mass_ineq_sub} in Theorem \ref{cpt_crit_thm_combined}. Therefore we know that $u \leq C$
for $u$ and $C$ as stated in Theorem \ref{cpt_crit_thm_combined}. 
Next, assume $u$ obtains $\sup_M u$ at a point $P$, then $\Delta u(P) \leq 0$. By \eqref{yamabeequ}, $u(P) \geq \sqrt{R_g(P)/(6K(P))}\geq c_0$. By the Harnack inequality, 
\begin{align}
\inf_M u\geq (1/c_1)u(P)\geq c_0/c_1 \geq 1/C, 
\end{align}
for sufficiently large $C$.
By standard elliptic estimates, we conclude that
$\Vert u\Vert_{C^{2,\alpha}(M)}\leq C$.
Furthermore, because for $p<3$, there always exist subcritical solutions for any $K$ (the proof on manifolds remains valid in the orbifold setting; see for example \cite{E-S, Lee-Parker}).  Take a sequence of subcritical solutions and let $p\to 3$, due to \eqref{comp_est}, they limit to a critical solution for $p=3$. Thus the second part of Theorem \ref{cpt_crit_thm_combined} is proved.

The first part of Theorem \ref{cpt_crit_thm_combined} can be proved similarly, by fixing $p_k=3$ in the above proof.
\end{proof}

\section{Variational methods}
\label{s:Energy}
Let $(M, g)$ be a compact Riemannian $n$-orbifold with singularities $\Sigma_\Gamma = \{(q_1,\Gamma_1),\cdots,(q_l,\Gamma_l)\}$ and positive scalar curvature $R_g$. Let $K$ be a positive smooth function on $M$. In this section, we will study equation \eqref{yamabeequ} using a variational method. Consider the energy functional
\begin{align}
\label{functional_J}
J_p(u, K, M) = \frac{\int_M (|\n u|^2 + c(n)R_g u^2 d)Vol_{g}}{(\int_M K |u|^{p+1}dVol_{g})^{\frac{2}{p+1}}},
\end{align}
for $1<p\leq \frac{n+2}{n-2}$ and $u\in W^{1,2}(M)$. Define the minimal energy to be
\begin{align}
E(p, K) = \inf_{u\in W^{1,2}(M)}J_p(u, K, M).
\end{align}
Let $Q(S^n)$ denote the Sobolev quotient of $S^n$, which is also the minimal energy on $S^n$ for $K \equiv 1$, that is, 
\begin{align}
\label{QS^n}
Q(S^n) = \inf_{\phi\in W^{1,2}(S^n)}J_{\frac{n+2}{n-2}}(\phi, 1, S^n) = \frac{n(n-2)}{4} Vol(S^n)^{2/n}.
\end{align}
The following theorem generalizes \cite[Proposition~1.1]{E-S} to the orbifold case, and also generalizes \cite[Theorem 3.1]{Akutagawa} from the case $K = constant$ to the case of variable $K$.
\begin{theorem}
\label{energy_obf}
Let $(M, g)$ be a compact Riemannian $n$-orbifold with positive scalar curvature and singularities $\Sigma_\Gamma = \{(q_1,\Gamma_1),\cdots,(q_l,\Gamma_l)\}$. Let $K$ be a positive smooth function on $M$. Define the modified maximum value of $K$
\begin{align}
\label{B_K}
B_K := \max\Big\{\sup_{x\in M}\{K(x)\}, \max_{1\leq i \leq l}\{|\Gamma_i|^{\frac{2}{n-2}}K(q_i)\}\Big\}.
\end{align}
Then the following inequality always holds
\begin{align}
\label{akubound0}
(B_K)^{\frac{n-2}{n}}E\Big(\frac{n+2}{n-2}, K\Big) \leq Q(S^n).
\end{align}
Furthermore, if
\begin{align}
\label{akubound}
(B_K)^{\frac{n-2}{n}}E\Big(\frac{n+2}{n-2}, K\Big) < Q(S^n),
\end{align}
then there exists a positive smooth solution $u$ of \eqref{yamabeequ} with $p=\frac{n+2}{n-2}$, such that 
\begin{align}
J_{\frac{n+2}{n-2}}(u, K, M)=E\Big(\frac{n+2}{n-2}, K\Big).
\end{align}
\end{theorem}
\begin{proof}
We first prove the inequality \eqref{akubound0}. The quantity $B_K$ is attained at either a singular point or a smooth point. 
Consider any point $q$ with  $|\Gamma|^{\frac{2}{n-2}}K(q) = B_K$,
where $\Gamma$ is the orbifold group at $q$ if $q$ is a singular point, or $\Gamma = \{e\}$ if $q$ is a smooth point. 
Take a conformal mapping $g_q = u_q^{\frac{4}{n-2}}g$ such that $g_q$ is the conformal normal metric centered at the point $q$. Consider the Green function with the power series expansion
\begin{align}
G_q = \frac{1}{4n(n-1)Vol(S^{n-1})}(r_q^{2-n} + H_q),
\end{align}
where $r_q$ is the geodesic distance from $q$ based on the metric $g_q$ and $H_q$ is the higher order term.
Take a family of test functions
\begin{align}
\label{testvarphi}
\varphi_{q, \lambda} = u_q\cdot \Bigg(\frac{\lambda}{1+\lambda^2(r_q^{2-n} + H_q)^{\frac{2}{2-n}}}\Bigg)^{\frac{n-2}{2}},
\end{align}
where $\lambda >0$. By \cite[Proposition 5.1]{M-M}, we have the estimation
\begin{align}
\label{Jest}
J_{\frac{n+2}{n-2}}(\varphi_{q,\lambda}, K, M) = \frac{c(n)\hat c_0}{|\Gamma|^{\frac{2}{n}}K(q)^{\frac{n-2}{n}}} + O(1/\lambda^2),
\end{align}
where
\begin{align}
\label{hatc_0}
\hat c_0 = 4n(n-1)\Big(\int_{\RR^n}\frac{dx}{(1+r^2)^n}\Big)^{\frac{2}{n}}.
\end{align}
Note that our functional $J_{\frac{n+2}{n-2}}$ is equal to $c(n)$ multiplied with the functional $J$ defined in \cite[Chapter~1]{M-M}. If $q$ is an orbifold point,  we estimate integrals in $J_{\frac{n+2}{n-2}}$ by lifting everything up to the universal cover near the orbifold point $q$. Therefore, compared to \cite[Proposition~5.1]{M-M}, our estimation \eqref{Jest} has an extra factor of $c(n)/|\Gamma|^{\frac{2}{n}}$. By a direct computation, we know that
\begin{align}
\label{c(n)}
c(n)\hat c_0 = Q(S^n).
\end{align}
Clearly, we have
\begin{align}
E\Big(\frac{n+2}{n-2}, K\Big) \leq \lim_{\lambda\to\infty}J_{\frac{n+2}{n-2}}(\varphi_{q,\lambda}, K, M) = \frac{Q(S^n)}{|\Gamma|^{\frac{2}{n}}K(q)^{\frac{n-2}{n}}}.
\end{align}
By our assumption, we have $|\Gamma|^{\frac{2}{n}}K(q)^{\frac{n-2}{n}} = B_K^{\frac{n-2}{n}}$, hence \eqref{akubound0} is proved.

We next show that the strict inequality \eqref{akubound} implies the existence result.  The proof is a modification of \cite[Theorem 3.1]{Akutagawa}. If we have proved the case that $M$ has only one singularity, the more general cases can be proved by an induction on the number of singularity points. Hence we may assume $M$ has only one singularity $(q, \Gamma)$. Let $X = M - \{q\}$. Note that
\begin{align}
E\Big(\frac{n+2}{n-2}, K\Big) = \inf_{u\in C_c^\infty(X)}J_{\frac{n+2}{n-2}}(u, K, X).
\end{align}
Let $B_\rho$ denote the open geodesic ball centered at $q$ of radius $\rho$. Define
\begin{align}
\label{Y_k_def}
Y_k \equiv \inf_{u\in C_c^\infty(X \setminus \overline{B_{1/k}})}J_{\frac{n+2}{n-2}}(u, K, X\setminus\overline{B_{1/k}}),
\end{align}
for $k\in\mathbb N^*$. It follows that
\begin{align}
Y_k \geq Y_{k+1} \geq Y_{k+2} \geq\cdots, \ \ \mathrm{and}\ \ \lim_{k\to\infty} Y_k= E\Big(\frac{n+2}{n-2}, K\Big).
\end{align}
By \eqref{akubound}, there exists a large integer $k_0$, such that
\begin{align}
(B_K)^{\frac{n-2}{n}}Y_k < Q(S^n)\ \mathrm{for\ any\ }k\geq k_0.
\end{align}
Due to \eqref{B_K}, it implies
\begin{align}
\label{E-S_prop1.1}
\bigg(\sup_{x\in X \setminus \overline{B_{1/k}}}\{K(x)\}\bigg)^{\frac{n-2}{n}}Y_k < Q(S^n)\ \mathrm{for\ any\ }k\geq k_0.
\end{align}
Note that on the manifold with boundary $(N,\p N)= (X - B_{1/k}, \p B_{1/k})$, when we apply integration by parts to any function in $C_c^\infty(X \setminus \overline{B_{1/k}})$, the boundary integral term always vanishes. As a result, the variational method used in \cite[Proposition 1.1]{E-S} remains valid here. Thus, \eqref{E-S_prop1.1} implies that for each $k\geq k_0$, there exists a non-negative $J_{\frac{n+2}{n-2}}(\cdot, K, X\setminus\overline{B_{1/k}})$-minimizer $u_k\in C^\infty(X - B_{1/k})$, such that
\begin{align}
\label{Lebesgue}
&J_{\frac{n+2}{n-2}}(u_k, K, X\setminus\overline{B_{1/k}}) = Y_k,\ \ \int_{X\setminus\overline{B_{1/k}}}Ku_k^{\frac{2n}{n-2}}dVol_g=1,\\
&u_k = 0\ \mathrm{on\ }\p B_{1/k}\ \ \mathrm{and}\ \ u_k > 0\ \mathrm{in\ } X\setminus\overline{B_{1/k}}.
\end{align}
Denote the zero extension of each $u_k$ to $M$ by also the same symbol $u_k$. Suppose the sequence $\{u_k\}$ has a uniform $C^0$-bound, i.e., there exists a constant $C>0$ such that $\Vert u_k\Vert_{C^0(M)} \leq C \mathrm{\ for\ }k\geq k_0$, then there exists a non-negative $J_{\frac{n+2}{n-2}}(\cdot, K, M)$-minimizer $u\in W^{1,2}(M)$ with $||u||_{C^0(M)}\leq C$, such that
\begin{align}
u_k\to u\mathrm{\ weakly\ in\ }W^{1,2}(M),\quad u_k\to u\mathrm{\ strongly\ in\ }L^2(M).
\end{align}
By \eqref{Lebesgue}, Lebesgue's bounded convergence theorem, and the above uniform $C^0$-bound, we have
\begin{align}
\int_{M}Ku^{\frac{2n}{n-2}}dVol_g = 1,
\end{align}
which further implies 
\begin{align}
u_k\to u\mathrm{\ strongly\ in\ }W^{1,2}(M).
\end{align}
By the standard elliptic estimate, we obtain that $u\in C^\infty(M)$. The maximum principle implies that $u>0$ everywhere on $M$; see for example \cite[Proposition~3.75]{Aubin}.

To complete the proof, it is sufficient to show a uniform $C^0$-bound for the sequence $\{u_k\}$. For each $k$, take the absolute maximum point $x_k\in X$ of $u_k$ and denote $M_k \equiv u_k(x_k)$. Taking a subsequence if necessary, there exists a point $\bar x\in M$ such that
\begin{align}
\lim_{k\to \infty}x_k = \bar x.
\end{align}
Suppose that there is not a uniform $C^0$-bound for $\{u_k\}$, that is $\lim_{k\to\infty}M_k = \infty$. There will be two cases.

Case 1: $\bar x = q$ (blow-up occurs at the singular point). In this case, we consider the universal cover of a small neighborhood around $q$. Let $\{\tilde x_k\}$ be a sequence of lifting-up points of $\{x_k\}$ in the same branch of the lifting-up space. Let $\tilde u_k$, $\tilde K$ denote the lifting-up functions of $u_k$, $K$, respectively. In the lifting-up space, for each $k$, let $\{\tilde x^i\}$ be a normal coordinate system in a small ball $B_\s(\tilde x_k)$ centered at each $\tilde x_k$. Consider the change of variables $\tilde y = M_k^{\frac{2}{n-2}}\tilde x$ and define the rescaled function $\tilde v_k(\tilde y) = M_k^{-1}\cdot \tilde u_k(M_k^{-\frac{2}{n-2}}\tilde y)$. By Theorem~2.1 of Chapter~5 in \cite{SchoenYaubook},
\begin{align}
\label{vktovconv}
\tilde v_k\to \tilde v\ \mathrm{in\ }C^2_{loc}(\RR^n),
\end{align}
where $\tilde v$ satisfies the equation
\begin{align}
\label{eucl_tilde}
-\Delta_0 \tilde v = E\Big(\frac{n+2}{n-2}, K\Big)K(q)\tilde v^{\frac{n+2}{n-2}}\ \mathrm{on}\ \RR^n.
\end{align}
Note that
\begin{align}
\int_{\RR^n}K(q)\tilde v(y)^{\frac{2n}{n-2}}dy= \lim_{r\to\infty}\int_{|\tilde y|\leq r}K(q)\tilde v(\tilde y)^{\frac{2n}{n-2}}d\tilde y.
\end{align}
For each $r>0$, by \eqref{vktovconv} and changing variables, we have
\begin{align}
\begin{split}
\int_{|\tilde y|\leq r}K(q)\tilde v(\tilde y)^{\frac{2n}{n-2}}d\tilde y =& \lim_{k\to\infty}\int_{|\tilde y|\leq r}\tilde K(M_k^{-\frac{2}{n-2}}\tilde y)\tilde v_k(\tilde y)^{\frac{2n}{n-2}}d\tilde y\\
=& \lim_{k\to\infty}\int_{|\tilde x|\leq M_k^{-\frac{2}{n-2}}r}\tilde K(\tilde x)\tilde u_k(\tilde x)^{\frac{2n}{n-2}}d\tilde x\\
\leq &|\Gamma|\cdot \lim_{k\to\infty}\int_{M}Ku_k^{\frac{2n}{n-2}}dVol_g= |\Gamma|,
\end{split}
\end{align}
which implies
\begin{align}
\label{denobound_tilde}
\int_{\RR^n}\tilde v^{\frac{2n}{n-2}}\leq |\Gamma|K(q)^{-1}.
\end{align}
Multiplying $\tilde v$ with \eqref{eucl_tilde} and integrating by parts, we obtain
\begin{align}
\int_{\RR^n}|\nabla \tilde v|^2 = E\Big(\frac{n+2}{n-2}, K\Big)K(q)\cdot\int_{\RR^n}\tilde v^{\frac{2n}{n-2}}.
\end{align}
Together with \eqref{denobound_tilde}, we have
\begin{align}
\begin{split}
 Q(S^n) &\leq J_{\frac{n+2}{n-2}}(\tilde v, 1, \RR^n) =
\frac{\int_{\RR^n}|\nabla \tilde v|^2}{(\int_{\RR^n}\tilde v^{\frac{2n}{n-2}})^{\frac{n-2}{n}}} 
\\&= E\Big(\frac{n+2}{n-2}, K\Big)K(q)\cdot\Big(\int_{\RR^n}\tilde v^{\frac{2n}{n-2}}\Big)^{\frac{2}{n}}\leq E\Big(\frac{n+2}{n-2}, K\Big)[|\Gamma|^{\frac{2}{n-2}}K(q)]^{\frac{n-2}{n}}.
\end{split}
\end{align}
Note that $q$ is a singular point on $M$, so $|\Gamma|^{\frac{2}{n-2}}K(q) \leq B_K$, which leads to a contradiction against~\eqref{akubound}.

Case 2: $\bar x\neq q$ (blow-up occurs at a smooth point). The argument is identical to Case 1 with $|\Gamma| = 1$. Hence
\begin{align}
\label{smoothblow-up}
E\Big(\frac{n+2}{n-2}, K\Big)K(\bar x)^{\frac{n-2}{n}} \geq Q(S^n).
\end{align}
Since $\bar x$ is a smooth point on $M$, $K(\bar x) \leq B_K$, which again leads to a contradiction against~\eqref{akubound}.
\end{proof}

\begin{proof}[Proof of Theorem~\ref{energy_existence}]
The proof will be based on Theorem \ref{energy_obf} and \cite[Proposition 5.1]{M-M}. Under the assumptions of Theorem \ref{energy_existence}, we simply denote the singular point $(q_{i_0}, \Gamma_{i_0})$ by $(q, \Gamma)$. Let $\varphi_{q,\lambda}$ be as defined in \eqref{testvarphi}. By \cite[Proposition 5.1]{M-M} and \eqref{c(n)}, we have the estimation
\begin{align}
J_{\frac{n+2}{n-2}}(\varphi_{q,\lambda}, K, M) = \frac{Q(S^n)}{|\Gamma|^{\frac{2}{n}}K(q)^{\frac{n-2}{n}}}\Bigg(1-\hat c_2\frac{\Delta K(q)}{K(q)\lambda^2} - \hat d_1 \left(\begin{array}{ll}\frac{H_q + O(\frac{\log \lambda}{\lambda^2})}{\lambda^2}&\mathrm{for\ }n=4\\ \frac{H_q}{\lambda^3}&\mathrm{for\ }n=5 \\ \frac{W_q\log\lambda}{\lambda^4}&\mathrm{for\ }n=6 \\0&\mathrm{for\ }n\geq 7 \\\end{array}\right) \Bigg)
\end{align}
up to error $O(1/\lambda^4)$ for large $\lambda$, where\footnote{The constant factor $\frac{2n}{n-2}$ was missing when the authors computed $\hat d_1$ in \cite[Proposition 5.1]{M-M}. We've corrected the mistake here.}
\begin{align}
\hat c_2 = \frac{\int_{\RR^n}\frac{r^2dx}{(1+r^2)^n}}{2n\int_{\RR^n}\frac{dx}{(1+r^2)^n}}\ \ \mathrm{and}\ \ \hat d_1 = \frac{2n\int_{\RR^n}\frac{r^ndx}{(1+r^2)^{n+1}}}{(n-2)\int_{\RR^n}\frac{dx}{(1+r^2)^n}}.
\end{align}
Moreover, in dimension 4, $\hat c_2 = 1/4$ and $\hat d_1 = 6$. Then, it is clear that under the following assumptions
\begin{equation}
\label{condH_q}
\begin{cases}
H_q + \frac{\Delta_g K(q)}{24 K(q)} > 0, & \mathrm{for\ }n=4,\\
\Delta_g K(q) > 0, & \mathrm{for\ }n \geq 5,
\end{cases}
\end{equation}
by choosing sufficiently large $\lambda$, we can get
\begin{align}
J_{\frac{n+2}{n-2}}(\varphi_{q,\lambda}, K, M) < \frac{Q(S^n)}{|\Gamma|^{\frac{2}{n}}K(q)^{\frac{n-2}{n}}}.
\end{align}
Note that we have $B_K=|\Gamma|^{\frac{2}{n-2}}K(q)$ under the assumption of Theorem~\ref{energy_existence}, hence the above inequality implies
\begin{align}
(B_K)^{\frac{n-2}{n}} E\Big(\frac{n+2}{n-2}, K\Big) < Q(S^n).
\end{align}
By Theorem \ref{energy_obf}, there exists a postive smooth solution of equation \eqref{yamabeequ}. Moreover, by \eqref{m=12A}, we know that condition \eqref{energy_exist_cond} is equivalent to condition \eqref{condH_q} in dimension 4.
\end{proof}

\section{LeBrun metrics}
\label{s:LeBrun}
Recall the LeBrun metric $(\mathcal O_{\mathbb{P}^1}(-n), g_{LEB(n)})$ in $\hr$-coordinates and its conformal compactification $(\check{\mathcal O}_{\mathbb{P}^1}(-n), \check g_{LEB(n)})$ from Section~\ref{lebrunsub}. On $\check{\mathcal O}_{\mathbb{P}^1}(-n)$, define $s = 1/\hr$ to be the inverted radial coordinate centered at the orbifold point $\check q$. Then
\begin{align}
\check g_{LEB(n)} =& \frac{1+s^2}{(1+n s^2)^3}ds^2 + \frac{s^2(s^2+1)}{(n s^2+1)^2}\Big[\s_1^2+\s_2^2 +\frac{(1+n s^2)}{(1+s^2)^2}\s_3^2\Big]\\
\label{leb_obf}
=&ds^2 + s^2(\s_1^2+\s_2^2+\s_3^2) + O(s^2)\ \ \mathrm{\ as \ }s \to 0.
\end{align}
Hence we may choose normal coordinates $\{x^i\}$ centered at $\check q$ such that $s = |x| + O(|x|^3)$. Also note that $g_{LEB(n)} = (n+\hr^2)^2\cdot \check g_{LEB(n)} = (s^{-2}+n)^2\cdot \check g_{LEB(n)}$ is scalar-flat, hence $g_{LEB(n)}$ is the conformal blow-up of $\check g_{LEB(n)}$ at the point $\check q$, as in Definition \ref{def_psi}. We will next study
\begin{align}
\label{sub_leb_conf}
\Delta_{\check g_{LEB(n)}}u - \frac{1}{6}R_{\check g_{LEB(n)}}u = -Ku^p,\ \ 1<p\leq 3.
\end{align}
In the $U(2)$-invariant case, Theorem \ref{cpt_crit_thm_combined} specializes to the following.

\begin{proposition}
\label{cpt_leb}
Let $\mathcal X_n$, $\mathcal X_{n,+}$, $\mathcal X_{n,0}$ and $\mathcal X_{n,-}$ be as defined in $\eqref{X_def}$ and $\eqref{X_decomp}$. On the orbifold $(\check{\mathcal{O}}_{\mathbb {P}^1}(-n),  \check g_{LEB(n)})$, for any $K\in\mathcal X_{n,+}\cup \mathcal X_{n,-}$, there exists some constant $C$, depending only on $\inf K$ and $\Vert K\Vert_{C^2}$, such that
\begin{align}
\label{comp_est_leb}
1/C\leq u\leq C\mathrm{\ \ and\ \ }\Vert u\Vert_{C^{2,\alpha}}\leq C,
\end{align}
for all U$(2)$-invariant solutions $u\in\mathcal X_n$ of \eqref{sub_leb_conf} with $p=3$, where $0<\alpha<1$. Moreover, if $K\in\mathcal X_{n,+}$,
then \eqref{comp_est_leb} holds for all  $1< 1+\varepsilon <p\leq 3$, where $C$ in addition depends on $\varepsilon$. Consequently, in this case there exists a $U(2)$-invariant solution
$u$ of \eqref{sub_leb_conf} with $p=3$. 
\end{proposition}

\begin{proof}
Assume $\{u_k\}\subset \mathcal X_n$ is a family of U$(2)$-invariant solutions of \eqref{sub_leb_conf} with corresponding exponents $p_k\to 3$ as $k\to\infty$. Assume $u_k$ blows up at a point $\bar x$ when $k\to\infty$. Suppose $\bar x$ is in the set $\{s = s_0\}$ for some $s_0>0$. Because $u_k$ is U$(2)$-invariant, $u_k$ blows up on the entire set $\{s=s_0\}$, which is either a hypersurface if $s_0<\infty$ or the $\mathbb{CP}^1$ component if $s_0=\infty$. However, either case is impossible, because Proposition \ref{iso} implies that blow-up points are isolated. Therefore, $u_k$ must blow up at the orbifold point $\check q$. It is clear that if $K\in \mathcal X_{n, +}$ implies \eqref{mass_ineq_sub}, and $K\in \mathcal X_{n, -}$ implies \eqref{mass_ineq}.  By the proof and statement of Theorem \ref{cpt_crit_thm_combined}, this proposition is proved.
\end{proof}
If $u(s)$ satisfies equation \eqref{leb_conf}, then
$u^2g_{\check g_{LEB(n)}}$ has scalar curvature $6 K(s)$. This latter condition is
 equivalent to $v^2g_{g_{LEB(n)}}$ having scalar curvature $6K(s)$, where
\begin{align}
\label{vu}
v(s) \equiv \frac{u(s)}{n + 1/s^2},
\end{align}
which is equivalent to
\begin{align}
\label{leb_ALE_conf}
\Delta_{g_{LEB(n)}} v = -K v^3.
\end{align}
Define another function space
\begin{align}
\mathcal Y_n = \Big\{h(s): (0,\infty) \rightarrow \RR_+ \ | \ h(s) = \frac{f(s)}{n + 1/s^2}, f(s)\in \mathcal X_n\Big\}.
\end{align}
There is an obvious bijection between the solution set of equation \eqref{leb_conf} in $\mathcal X_n$ and the solution set of equation \eqref{leb_ALE_conf} in $\mathcal Y_n$.

\begin{proposition}
\label{O(-2)}
On $(\mathcal{O}_{\mathbb {P}^1}(-2),  g_{LEB(2)})$, assume $K(s)\in\mathcal X_2$ satisfies
\begin{align}
\label{K'<0}
K'(s) \leq 0\ \ \mathrm{for\ any\ } s\geq 0,
\end{align}
then there is no solution in $\mathcal Y_2$ that solves \eqref{leb_ALE_conf}. Consequently, there is no solution in $\mathcal X_2$ that solves \eqref{leb_conf}.
\end{proposition}
\begin{proof}
Suppose $v(s) \in\mathcal Y_2$ solves equation \eqref{leb_ALE_conf}. It implies that $\check g = v^2g_{LEB(2)}$ is a compact orbifold metric with $R_{\check g}=6 K(s)$, and it has an orbifold point $\check q$ at $s = 0$. Denote this orbifold by $\check M$. Define $\hat v(\hr) = v(1/\hr)$ and $\hat K(\hr) = K(1/\hr)$. Then in the $\hr$-coordinate, we have
\begin{align}
\frac{n + 3\hr^2}{\hr(1+\hr^2)}\hat v'(\hr) + \frac{n +\hr^2}{1+\hr^2}\hat v''(\hr) = -\hat K(\hr)\hat v(\hr)^3.
\end{align}
Denote the left hand side of the above equation by LHS. We observe that
\begin{align}
\mathrm{LHS} \leq 0\ \ \mathrm{and}\ \ \mathrm{LHS} = \frac{[(n\hr+\hr^3)\hat v'(\hr)]'}{\hr(1+\hr^2)}.
\end{align}
It follows that
\begin{align}
[(n\hr+\hr^3)\hat v'(\hr)]'\leq 0,
\end{align}
which implies that
\begin{align}
(n\hr+\hr^3)\hat v'(\hr) \leq  0.
\end{align}
Thus $\hat v'(\hr) \leq 0$ for any $\hr\geq 0$, which implies that
\begin{align}
\label{v'>0}
v'(s) \geq 0,\ \ \mathrm{for\ any\ } s\geq 0.
\end{align}
Similar to the proof of \cite[Theorem 1.3]{monopole_V}, letting $E$ denote the traceless Ricci tensor, by the conformal transformation formula for $E$ and the fact that $g_{LEB(2)}$ is Ricci-flat, we have
\begin{align}
E_{\check g} = v^{-1}(-2\nabla^2 v + (\Delta v/2)\check g),
\end{align}
where $\nabla$ and $\Delta$ are respect to metric $\check g$. Using the argument of \cite{Obata}, integrating $v |E_{\check g}|^2$ on $\check M$, we obtain
\begin{align}
\begin{split}
\int_{\check M}v |E_{\check g}|^2dVol_{\check g} = &\int_{\check M}v E^{ij}_{\check g}\Big\{v^{-1}(-2\nabla^2 v + (\Delta v/2)\check g)_{ij} \Big\}dVol_{\check g} \\
=& -2\int_{\check M} E^{ij}_{\check g}(\nabla^2 v)_{ij}dVol_{\check g}\\
=& -2\lim_{\varepsilon\to 0}\int_{\check M\setminus B_\varepsilon(\check q)} E^{ij}_{\check g}(\nabla^2 v)_{ij}dVol_{\check g},
\end{split}
\end{align}
where $B_{\varepsilon}(\check q)$ denotes the geodesic ball centered at $\check q$ with radius $\varepsilon$, and the indices $i=0,1,2,3$ denote the directions $ds, \s_1, \s_2, \s_3$ in the moving coframe. Integrating by parts, we have
\begin{align}
\int_{\check M}v |E_{\check g}|^2dVol_{\check g} =-2\lim_{\varepsilon\to 0}\Big(\int_{\p B_\varepsilon(\check q)} E^{ij}_{\check g}(\nabla v)_i\nu_jd\s_{\check g} - \int_{\check M\setminus B_\varepsilon(\check q)} \nabla_jE^{ij}_{\check g}\cdot (\nabla v)_idVol_{\check g}\Big).
\end{align}
By \eqref{vu}, it is clear that $v(s) = O(s^2)$ near $s = 0$. Also because the curvature and volume term for $\check g$ are bounded near $s=0$, the first integral term on the right hand side goes to zero as $\varepsilon\to 0$. For the second integral term on the right hand side, since $v(s)$ is a radial function, the only nonvanishing component of $(\nabla v)_i$ is $(\nabla v)_0 = v'(s)$ along the $\p/\p s$ direction. On the other hand, by the Bianchi identity,
\begin{align}
\nabla_jE_{\check g}^{ij} = \nabla_j\Big(Ric_{\check g}^{ij}-\frac{1}{4}R_{\check g}\cdot  \check g^{ij}\Big) = \frac{1}{2}\nabla^i R_{\check g} - \frac{1}{4}\nabla^i R_{\check g} = \frac{3}{2}\nabla^i K.
\end{align}
Since $K(s)$ is a radial function, the only nonvanishing component of $(\nabla K)^i$ is $(\nabla K)^0 = K'(s)$ along the $ds$ direction. Hence
\begin{align}
\label{bianchi}
\nabla_jE^{ij}_{\check g} = \begin{cases}
\frac{3}{2}K'(s), & i=0,\\
0, & i=1,2,3.
\end{cases}
\end{align}
It follows
\begin{align}
\int_{\check M}v |E_{\check g}|^2dVol_{\check g} =& \lim_{\varepsilon\to 0} \int_{\varepsilon}^\infty C(s)K'(s)v'(s)ds,
\end{align}
where $C(s)$ is a positive function in $s$, depending on the volume term of $\check g$ at each $s$. By \eqref{K'<0} and \eqref{v'>0}, we know
\begin{align}
\lim_{\varepsilon\to 0} \int_{\varepsilon}^\infty C(s)K'(s)v'(s)ds \leq 0,
\end{align}
hence
\begin{align}
\int_{\check M}v |E_{\check g}|^2dVol_{\check g} \leq 0,
\end{align}
which implies $E_{\check g} = 0$. By \eqref{bianchi}, we know $K'(s) \equiv 0$, so that $K = constant$. By \cite[Theorem~1.3]{monopole_V}, such a solution $v$ does not exist.
\end{proof}
For each $n$, we define another radial coordinate by
\begin{align}
t(\hr) = \log\Big(\frac{n+\hr^2}{\hr^2}\Big)\ \ \mathrm{and\ its\ inverse}\ \ \hr(t)=\sqrt{\frac{n}{e^{t}-1}}.
\end{align}
Hence
\begin{align}
\label{st}
t(s) = \log(n s^2+1)\ \ \mathrm{and}\ \ s(t) = \sqrt{\frac{e^{t}-1}{n}}.
\end{align}
We have defined three radial coordinates $\hr, s, t\in [0,\infty]$. It is not hard to see that $1/\hr$, $s$ and $t$ are monotonic to each other. By basic computations, on the ALE manifold $\mathcal O_{\mathbb{P}^1}(-n)$, we have
\begin{align}
\label{Delta_t}
\Delta_{g_{LEB(n)}} =& \frac{n + 3\hr^2}{\hr(1+\hr^2)}\frac{\p}{\p\hr} + \frac{n +\hr^2}{1+\hr^2}\frac{\p^2}{\p\hr^2}
= \frac{4(1-e^{-t})^3}{e^{-t}(1+(n-1)e^{-t})}\frac{\p^2}{\p t^2}.
\end{align}
Next, we prove the following non-existence result. 
\begin{theorem}
\label{lebrun_thm_2}
For any $n\in\mathbb N^*$, on $(\check{\mathcal{O}}_{\mathbb {P}^1}(-n),  \check g_{LEB(n)})$, there is no solution in $\mathcal X_n$ for equation \eqref{leb_conf} with
\begin{align}
\label{Kn-}
K = K_{n, -}(s) = \frac{2+n s^2}{n + n s^2}K_{2,-}(s),
\end{align}
where $K_{2,-}(s)\in\mathcal X_n$ is any monotonically decreasing function in $s$. Moreover, we know
\begin{align}
K = K_{n, -}(s)\in \mathcal X_{n, -}\mathrm{\ if\ }K_{2,-}''(0) < 0.
\end{align}
In particular, for $n\geq 2$, it implies that there is no $U(2)$-invariant solution for the Yamabe Problem $K = constant$.  
\end{theorem}
\begin{proof}
For such a $K = K_{n, -}(s)$, suppose there is a $u_n(s)\in \mathcal X_n$ that solves equation \eqref{leb_conf}, then $v_n(s)=u_n(s)/(n+s^{-2})\in \mathcal Y_n$ solves equation \eqref{leb_ALE_conf}. Define
\begin{align}
\label{barKn-}
\bar v_n(t) = v_n\Big(\sqrt{\frac{e^{t}-1}{n}}\Big),\ \ \bar K_{n,-}(t) = K_{n,-}\Big(\sqrt{\frac{e^{t}-1}{n}}\Big).
\end{align}
By \eqref{st} and \eqref{Delta_t}, we know \eqref{leb_ALE_conf} is equivalent to the following equation
\begin{align}
\bar v_n''(t) = -\bar K_{n,-}(t)\cdot \frac{e^{-t}(1+(n-1)e^{-t})}{4(1-e^{-t})^3} \bar v_n(t)^3.
\end{align} 
It implies
\begin{align}
\label{transform_equ}
\bar v_n''(t) = -\frac{(1+(n-1)e^{-t})\bar K_{n,-}(t)}{1+e^{-t}}\cdot \frac{e^{-t}(1+e^{-t})}{4(1-e^{-t})^3} \bar v_n(t)^3.
\end{align}
Define
\begin{align}
\label{barK2-}
\bar K_{2,-}(t) = \frac{(1+(n-1)e^{-t})\bar K_{n,-}(t)}{1+e^{-t}},
\end{align}
then \eqref{transform_equ} becomes
\begin{align}
\label{beta_t_equ}
\bar v_n''(t) = -\bar K_{2,-}(t)\cdot \frac{e^{-t}(1+e^{-t})}{4(1-e^{-t})^3} \bar v_n(t)^3,
\end{align}
which can be viewed as \eqref{leb_ALE_conf} on $(\mathcal{O}_{\mathbb {P}^1}(-2),  g_{LEB(2)})$ in the $t$-coordinate, with $K=\bar K_{2,-}(t)$ and $v=\bar v_n(t)$. To transform $\bar v_n(t)$ and $\bar K_{2,-}(t)$ back to $s$-coordinate on $\mathcal{O}_{\mathbb {P}^1}(-2)$, using \eqref{st} with $n=2$, we define
\begin{align}
v_2(s) = \bar v_n(\log(2 s^2+1))\ \ \mathrm{and}\ \ K_2(s) = \bar K_{2,-}(\log(2 s^2+1)).
\end{align}
Thus we know that on $\mathcal{O}_{\mathbb {P}^1}(-2)$, $v_2(s)\in\mathcal Y_2$ solves \eqref{leb_ALE_conf} with $K = K_2(s)\in \mathcal X_2$. By combining \eqref{Kn-}, \eqref{barKn-} and \eqref{barK2-}, it is not hard to see that
\begin{align}
K_2(s) = \frac{n+2s^2}{2+2s^2}\bar K_{n,-}(\log(2 s^2+1)) =  \frac{n+2s^2}{2+2s^2}K_{n,-}\Big(\sqrt{\frac{2}{n}}s\Big) = K_{2,-}\Big(\sqrt{\frac{2}{n}}s\Big).
\end{align}
Our assumption that $K_{2,-}(s)$ is a monotonically decreasing function in $s$ implies $K_2'(s)\leq 0$ for all $s$, which leads to a contradiction against Proposition \ref{O(-2)}.

For the last part, assume $K(s)\in \mathcal X_n$. Near $s=0$, 
$K$ has a power series expansion
\begin{align}
K(s) = K(0) + \frac{1}{2}K''(0)s^2 + O(s^4).
\end{align}
By \eqref{trans_lap} and \eqref{leb_obf},
\begin{align}
\Delta_{\check g_{LEB(n)}}K(0) = \Delta_{g_{\RR^4/\Gamma_n}}K(0) = K''(0) + \frac{3}{s}K'(s)\Big|_{s=0} = 4K''(0).
\end{align}
So, $K_{n,-}(s)\in \mathcal X_{n, -}$ is equivalent to 
\begin{align}
\frac{K_{n,-}''(0)}{K_{n,-}(0)} < n - 2.
\end{align}
If $K_{2,-}''(0)<0$, then we have
\begin{align}
\frac{K_{n,-}''(0)}{K_{n,-}(0)} = \frac{\frac{d^2}{ds^2}\Big( \frac{2+n s^2}{n + n s^2}K_{2,-}(s)\Big)}{ \frac{2+n s^2}{n + n s^2}K_{2,-}(s)}\Bigg|_{s=0} = n-2 + \frac{K_{2,-}''(0)}{K_{2,-}(0)} < n-2.
\end{align}
It follows $K_{n,-}(s)\in \mathcal X_{n, -}$.
\end{proof}

\begin{proof}[Proof of Theorem \ref{lebrun_thm}]
By Proposition \ref{cpt_leb}, and homotopy invariance of the Leray-Schauder degree, $\deg(F_{p,K,n},\Omega_{\Lambda,n}, 0)$ is equal to a constant in either case (1) or (2) in Theorem~\ref{lebrun_thm}. For any $n\in \mathbb N^*$, Theorem~\ref{lebrun_thm_2} implies that
\begin{align}
\deg(F_{3,K,n},\Omega_{\Lambda,n}, 0)= 0,
\end{align}
for $K\in \mathcal X_{n,-}$
satisfying the assumptions of Theorem~\ref{lebrun_thm_2}.
Homotopy invariance implies that the degree is zero for any $K\in \mathcal X_{n,-}$, which proves (2) of Theorem~\ref{lebrun_thm}. Part (1) of Theorem~\ref{lebrun_thm} is proved by a standard subcritical degree counting argument; see \cite{Schoen_csc}. 
\end{proof}

\bibliographystyle{amsalpha}
\bibliography{JV}

\end{document}